\documentclass[11pt]{amsart}
\usepackage{setspace}
\usepackage{graphicx}
\usepackage{geometry}
\usepackage{multirow}
\usepackage{commath}
\usepackage{mdframed}
\usepackage{amsmath, amsthm}
\usepackage{thmtools}
\usepackage{bbm}
\usepackage{amssymb}
\usepackage{array}
\newcolumntype{M}[1]{>{\centering\arraybackslash}m{#1}}
\usepackage{mathtools}

\usepackage{amsthm}
\usepackage{float}
\usepackage{caption}
\usepackage{color}
\usepackage{enumerate}
\usepackage{tikz}
\usepackage{hyperref}
\hypersetup{
   colorlinks = true, 
   linkcolor = blue, 
   filecolor = magenta,      
   urlcolor = black, 
}
\usepackage{url}
\expandafter\def\expandafter\UrlBreaks\expandafter{\UrlBreaks
  \do\a\do\b\do\c\do\d\do\e\do\f\do\g\do\h\do\i\do\j
  \do\k\do\l\do\m\do\n\do\o\do\p\do\q\do\r\do\s\do\t
  \do\u\do\v\do\w\do\x\do\y\do\z\do\A\do\B\do\C\do\D
  \do\E\do\F\do\G\do\H\do\I\do\J\do\K\do\L\do\M\do\N
  \do\O\do\P\do\Q\do\R\do\S\do\T\do\U\do\V\do\W\do\X
  \do\Y\do\Z}
\usepackage{setspace}
\usepackage{systeme}
\usepackage{tabularx}
\usepackage{cite}
\usepackage{etoolbox}
\usepackage[british]{babel}
\usepackage[symbol, flushmargin]{footmisc}

\newtheorem{theorem}{Theorem}[section]
\newtheorem{lemma}[theorem]{Lemma}

\newtheorem{question}[theorem]{Question}
\newtheorem{problem}[theorem]{Problem}

\newtheorem{claim}{Claim}[theorem]
\newtheorem*{claim*}{Claim}
\newtheorem*{theorem*}{Theorem}
\newtheorem*{prop*}{Proposition}
\newtheorem*{lemma*}{Lemma}
\newtheorem*{keyobservation*}{Key Observation}
\newtheorem*{conjecture*}{Conjecture}
\newtheorem*{mainthm*}{Main Theorem}
\numberwithin{equation}{section}
\theoremstyle{definition}

\newtheorem*{notation*}{Notation}
\newtheorem{defn}[theorem]{Definition}

\newtheorem{remark}[theorem]{Remark}

\newcommand{\R}[0]{\mathbb{R}}
\newcommand{\Q}[0]{\mathbb{Q}}
\newcommand{\N}[0]{\mathbb{N}}
\newcommand{\Z}[0]{\mathbb{Z}}
\newcommand{\M}[0]{\mathcal{M}}

\newcommand{\poly}[0]{\text{poly}}
\newcommand{\timesdots}[0]{\times\cdots\times}
\counterwithout{equation}{section}
\title[Zarankiewicz bounds from distal regularity lemma]{Zarankiewicz bounds from distal regularity lemma}
\author{Mervyn Tong}
\address{School of Mathematics, University of Leeds, Leeds LS2 9JT, United Kingdom}
\address{Correspondence: Department of Pure Mathematics and Mathematical Statistics, Centre for Mathematical Sciences, Wilberforce Road, Cambridge CB3 0WB, United Kingdom}
\email{hwmt3@cam.ac.uk}
\date{\today}
\begin{document}
\begin{abstract}
    Since K\H{o}v\'ari, S\'os, and Tur\'an proved upper bounds for the Zarankiewicz problem in 1954, much work has been undertaken to improve these bounds, and some have done so by restricting to particular classes of graphs. In 2017, Fox, Pach, Sheffer, Suk, and Zahl proved better bounds for semialgebraic binary relations, and this work was extended by Do in the following year to arbitrary semialgebraic relations. In this paper, we show that Zarankiewicz bounds in the shape of Do's are enjoyed by all relations satisfying the distal regularity lemma, an improved version of the Szemer\'edi regularity lemma satisfied by relations definable in distal structures (a vast generalisation of o-minimal structures).
\end{abstract}
\maketitle
\section{Introduction}\label{sectionintro}
A classical problem in graph theory is the Zarankiewicz problem, which asks for the maximum number of edges a bipartite graph with $n$ vertices in each class can have if it omits $K_{u,u}$, the complete bipartite graph with $u$ vertices in each class. In 1954, K\H{o}v\'ari, S\'os, and Tur\'an \cite{kovarisosturan} gave an upper bound of $O_u(n^{2-1/u})$. Remarkably, this remains the tightest known upper bound, although sharpness has only been proven for $u\in \{2,3\}$. In 2017, Fox, Pach, Sheffer, Suk, and Zahl \cite{fox} observed that this bound can be improved if the graph is semialgebraic.
\begin{theorem}[{Fox--Pach--Sheffer--Suk--Zahl \cite[Theorem 1.1]{fox}}]\label{fox}
    Let $E(x, y)$ be a semialgebraic relation on $\R$ with description complexity at most $t$. Let $d_1:=|x|$ and $d_2:=|y|$. Then, for all finite $P\subseteq \R^x$ and $Q\subseteq \R^y$ with $m:=|P|$ and $n:=|Q|$, if $E(P,Q)$ is $K_{u,u}$-free, then for all $\varepsilon>0$ we have
    \[|E(P,Q)|\ll_{u, d_1, d_2, t, \varepsilon}
    \begin{cases}
        m^{\frac{2}{3}}n^{\frac{2}{3}}+m+n&\text{if }d_1=d_2=2,\\
        m^{\frac{d_2(d_1-1)}{d_1d_2-1}+\varepsilon}n^{\frac{d_1(d_2-1)}{d_1d_2-1}}+m+n&\text{otherwise}.
    \end{cases}\]
\end{theorem}

The graph theorist naturally asks if these results can be generalised to $k$-uniform $k$-partite hypergraphs (henceforth, a \textit{$k$-graph} is a $k$-uniform hypergraph). Erd\H{o}s led the way in 1964 \cite{erdoshypergraphs}, generalising the result of K\H{o}v\'ari et al: a $K_{u,...,u}$-free $k$-partite $k$-graph with $n$ vertices in each class has $O_u(n^{k-1/u^{k-1}})$ edges. In 2018, Do \cite{do} generalised Theorem \ref{fox}, improving Erd\H{o}s' bounds for semialgebraic hypergraphs.

\begin{theorem}[{Do \cite[Theorem 1.7]{do}}]\label{do}
    Let $E(x_1, ..., x_k)$ be a semialgebraic relation on $\R$ with description complexity at most $t$. Let $d_i:=|x_i|$. Then, for all finite $P_i\subseteq \R^{x_i}$ with $n_i:=|P_i|$, if $E(P_1, ..., P_k)$ is $K_{u,...,u}$-free, then for all $\varepsilon>0$ we have
    \[|E(P_1, ..., P_k)|\ll_{u, \bar{d}, t, \varepsilon}F^\varepsilon_{\bar{d}}(\bar{n}),\]
    where $\bar{d}:=(d_1, ..., d_k)$ and $\bar{n}:=(n_1, ..., n_k)$.
\end{theorem}
The function $F^\varepsilon_{\bar{d}}$ will be defined in Definition \ref{Fdefn}, but for now we merely note that, when $d_1=\cdots=d_k=:d$ and $n_1=\cdots=n_k=:n$,
\[F^\varepsilon_{\bar{d}}(\bar{n})\ll_k n^{k-\frac{k}{(k-1)d+1}+k\varepsilon}.\]

In this paper, we prove an analogue of Theorem \ref{do} for a much larger class of relations, namely, relations satisfying the \textit{distal regularity lemma}. This is an improved version of the Szemer\'edi regularity lemma, in which the sizes of the partitions are polynomial in the reciprocal of the error, and the good cells are not just regular but homogeneous (that is, a clique or an anti-clique). Collecting the degrees of the polynomials into a \textit{strong distal regularity tuple} $\bar{c}$, we state our main theorem.

\begin{mainthm*}[Theorem \ref{mainabstractbound}]\label{mainthm}
    Let $E(x_1, ..., x_k)$ be a relation on a set $M$, with strong distal regularity tuple $\bar{c}=(c_1, ..., c_k)\in\R_{\geq 1}^k$ and coefficient $\lambda$. For all finite $P_i\subseteq M^{x_i}$ with $n_i:=|P_i|$, if $E(P_1, ..., P_k)$ is $K_{u, ..., u}$-free, then for all $\varepsilon>0$,
    \[|E(P_1, ..., P_k)|\ll_{u,\bar{c},\lambda,\varepsilon}F^\varepsilon_{\bar{c}}(n_1, ..., n_k).\]
\end{mainthm*}

Here, the function $F^\varepsilon_{\bar{c}}$ is precisely the function $F^\varepsilon_{\bar{d}}$ appearing in Theorem \ref{do}, but with $\bar{c}$ in place of $\bar{d}$ (as a tuple of dummy variables). The definition of the \textit{coefficient} is unimportant for this discussion, so we refer the reader to Definition \ref{tupledefn}.

The distal regularity lemma is so named because it is satisfied by all relations definable in \textit{distal} structures, which is a vast generalisation of o-minimal structures \cite{simondistal}. Examples of distal structures that are not o-minimal include the field of p-adic numbers \cite{simondistal}, Presburger arithmetic \cite{regularitylemma}, and expansions of Presburger arithmetic by predicates such as $\{2^n: n\in \N\}$, $\{n!: n\in \N\}$, and the set of Fibonacci numbers \cite{mypaper}.

Thus, our main theorem joins a parade of combinatorial properties that have been shown to hold in distal structures in the last decade. It supports the postulate by Chernikov, Galvin, and Starchenko in \cite{cuttinglemma} that `distal structures provide the most general natural setting for investigating questions in ``generalised incidence combinatorics''', where they proved an analogue of Theorem \ref{fox} for binary relations definable in a distal structure. It also motivates the following problem.

\begin{problem}
    Compute (strong) distal regularity tuples for relations satisfying the distal regularity lemma, such as those definable in a distal structure.
\end{problem}

It is worth pointing out that our main theorem applies to all relations satisfying the distal regularity lemma, which form a strictly larger class of relations than those definable in a distal structure --- see Theorem \ref{bays}. However, this does not constitute a refutation of the postulate by Chernikov, Galvin, and Starchenko, as distality remains the most general first-order property in the literature which guarantees that all relations in the structure satisfy the distal regularity lemma. See Section \ref{sectiondistality} for a discussion on distality.

This paper presents regularity lemmas as a means of obtaining Zarankiewicz bounds, an approach also adopted in \cite{janzerpohoata}. Improvements on the Szemer\'edi regularity lemma have been made in various contexts, such as for stable graphs \cite{malliarisshelah} and for graphs with bounded VC-dimension \cite{nipregularitylemma}. Following our main theorem, it is natural to ask the following question.
\begin{question}
    Which other variants of the Szemer\'edi regularity lemma give rise to improved Zarankiewicz bounds?
\end{question}
\subsection{The semialgebraic case}
Recently, Tidor and Yu \cite{hans} proved that if $E(x_1, ..., x_k)$ is a semialgebraic relation on $\mathbb{R}$, then $(|x_1|, ..., |x_k|)$ is a distal regularity tuple for $E$, and so $(|x_1|+1, ..., |x_k|+1)$ is a strong distal regularity tuple for $E$, where the coefficient is a function of the description complexity of $E$. We refer the reader to Definition \ref{tupledefn} for a precise definition of (strong) distal regularity tuples, but for now we emphasise that the word `strong' refers to requiring equipartitions in the distal regularity lemma.

Thus, if the assumption in our main theorem can be weakened so that $\bar{c}$ is only required to be a distal regularity tuple (that is, the corresponding partitions need not be equipartitions), their result can be combined with ours to recover Theorems \ref{fox} and \ref{do}. In Section \ref{sectionbinary}, we show that this assumption can indeed be so weakened when $E$ is a binary relation, thus recovering Theorem \ref{fox}. We would like to do likewise for an arbitrary relation.

\begin{question}
    For an arbitrary relation $E(x_1, ..., x_k)$, can the assumption in our main theorem be weakened so that $\bar{c}$ is only required to be a distal regularity tuple?
\end{question}

Another way to recover Theorem \ref{do} would be to answer the following question positively.
\begin{question}
    For a semialgebraic relation $E(x_1, ..., x_k)$ on $\mathbb{R}$, is $(|x_1|, ..., |x_k|)$ a strong distal regularity tuple for $E$?
\end{question}
We note however that, in \cite{hans}, Tidor and Yu also proved infinitesimally improved versions of Theorems \ref{fox} and \ref{do}, which our present methods are not able to achieve.
\subsection{Structure of the paper}
In Section \ref{sectionregularitytuples}, we introduce the notion of (strong) distal regularity tuples and prove some of their basic properties. In Section \ref{sectionbinary}, we prove a stronger version of the main theorem in the case where the relation is binary, and in Section \ref{sectionmain}, we prove the theorem in full. Finally, in Section \ref{sectiondistality}, we apply the theorem to distal structures, providing a brief commentary on distality for the uninitiated.
\subsection{Notation and basic definitions}
In this paper, all logical structures are first-order. We will often use $\mathcal{M}$ to denote a structure and $M$ to denote its underlying set.
\subsubsection{Indexing}\label{subsubsecindexing}
For $k\in\N^+$, $[k]:=\{1, ..., k\}$.

Let $\bar{c}=(c_1, ..., c_k)$ be a $k$-tuple. For $I\subseteq [k]$ enumerated in increasing order by $i_1, ..., i_l$, let $\bar{c}_I$ denote the $l$-tuple $(c_{i_1}, ..., c_{i_l})$. For $i\in [k]$, $\bar{c}_{\neq i}:=\bar{c}_{[k]\setminus\{i\}}$.

\subsubsection{Relations}

Let $M$ be a set and $k\in\N^+$. For tuples of variables $x_1, ..., x_k$, a \textit{relation} $E(x_1, ..., x_k)$ on $M$ is a subset of $M^{|x_1|}\timesdots M^{|x_k|}$. We will often drop the absolute value signs and write $M^{x_1}$ for $M^{|x_1|}$, and so on.

For $a_i\in M^{x_i}$, $E(a_1, ..., a_k)$ is defined to mean $(a_1, ..., a_k)\in E(x_1, ..., x_k)$. For $P_i\subseteq M^{x_i}$, $E(P_1, ..., P_k)$ denotes the induced $k$-subgraph of $E(x_1, ..., x_k)$ on $P_1\timesdots P_k$; that is,
\[E(P_1, ..., P_k)=\{(a_1, ..., a_k)\in P_1\timesdots P_k: E(a_1, ..., a_k)\}\subseteq P_1\timesdots P_k.\]
Say that $P_1\timesdots P_k$ is \textit{$E$-homogeneous} if $E(P_1, ..., P_k)=P_1\timesdots P_k$ or $\emptyset$.

For $P_i\subseteq M^{x_i}$ and $b_i\in M^{x_i}$, define $E(b_1, P_2, ..., P_k)$ to be the $(k-1)$-graph
\[\{(a_2, ..., a_k)\in P_2\timesdots P_k: E(b_1, a_2, ..., a_k)\}\]
on $P_2\timesdots P_k$, and we similarly define $E(P_1, ..., P_{i-1}, b_i, P_{i+1}, ..., P_k)$ for all $i\in [k]$.
\subsubsection{Covers and partitions}

Given a set $X$, a collection $(X_1, ..., X_k)$ of subsets of $X$ is said to \textit{cover} $X$ if $X=X_1\cup\cdots\cup X_k$. If, additionally, $X_1, ..., X_k$ are pairwise disjoint, we say that they \textit{partition} $X$; we will often write this partition as $X=X_1\sqcup\cdots\sqcup X_k$. A cover $X=X_1\cup\cdots\cup X_k$ \textit{refines} a cover $X=Y_1\cup\cdots\cup Y_l$ if, for all $i\in [k]$, there is $j\in [l]$ such that $X_i\subseteq Y_j$.

A partition $X_1\sqcup\cdots\sqcup X_k$ of a finite set $X$ is said to be an \textit{equipartition} if, for all $i,j\in [k]$, $||X_i|-|X_j||\leq 1$.

\subsubsection{Asymptotics}
Let $D,E$ be sets and $f(x,y),g(x,y): D\times E\to\R_{\geq 0}$, where $x,y$ are tuples of variables. Write $f(x,y)=O_x(g(x,y))$, $f(x,y)\ll_x g(x,y)$, or $g(x,y)\gg_x f(x,y)$ if there is $C=C(x): D\to \R_{\geq 0}$ such that $f(x,y)\leq C g(x,y)$ for all $x\in D$ and $y\in E$.

Let $h(y): E\to\R_{\geq 0}$. Write $f(x,y)\leq \text{poly}_x(h(y))$ if there is $C=C(x): D\to\R_{\geq 0}$ such that $f(x,y)\leq Ch(y)^C$ for all $x\in D$ and $y\in E$.

\subsection{Acknowledgements}
We thank Pantelis Eleftheriou for his consistent guidance and mentorship, and for suggesting this problem to us. We thank Martin Bays, Dugald Macpherson, and Vincenzo Mantova for their insightful comments and suggestions on the version of this paper contained in our PhD thesis, with special thanks to Martin for our discussions on Theorem \ref{bays}. We thank the referee for taking the time to review this paper and offer helpful suggestions.

\textit{Soli Deo gloria.}

\subsection{Funding statement}
The majority of the research in this paper was conducted as part of the author’s PhD, supported by a scholarship funded by the School of Mathematics at the University of Leeds. During the final stages of the revision process, the author was employed at the University of Cambridge and supported by Julia Wolf's Open Fellowship from the UK Engineering and Physical Sciences Research Council (EP/Z53352X/1).

\section{(Strong) distal regularity tuples}\label{sectionregularitytuples}
We begin by defining the notion of regularity for a bipartite graph.
\begin{defn}
    Let $M$ be a set, and let $E(x,y)$ be a relation on $M$. For finite $A\subseteq M^x$ and $B\subseteq M^y$, write
    \[d(A,B):=\frac{|E(A,B)|}{|A||B|}.\]
    
    Let $P\subseteq M^x$ and $Q\subseteq M^y$ be finite. For $\delta>0$, say that the bipartite graph $E(P,Q)$ is \textit{$\delta$-regular} if, for all $A\subseteq P$ and $B\subseteq Q$ with $|A|\geq \delta |P|$ and $|B|\geq \delta |Q|$, $|d(A,B)-d(P,Q)|\leq \delta$.
\end{defn}
In 1978, Szemer\'edi proved the following celebrated regularity lemma.
\begin{theorem}[Szemer\'edi, 1978 \cite{szemerediregularitylemma}]\label{szemeredi}
    Let $M$ be a set, and let $E(x,y)$ be a relation on $M$. For all $\delta>0$, there is $K\in\N$ such that the following holds.

    Let $P\subseteq M^x$ and $Q\subseteq M^y$ be finite. Then there are partitions $P=A_1\sqcup \cdots \sqcup A_K$ and $Q=B_1\sqcup \cdots\sqcup B_K$, and an index set $\Sigma\subseteq [K]^2$ of `bad cells', such that
    \begin{enumerate}[(a)]
        \item `Meagre bad cells': $\sum_{(i,j)\in\Sigma} |A_i\times B_j|\leq \delta |P\times Q|$; and
        \item `$\delta$-regular good cells': for all $(i,j)\in [K]^2\setminus \Sigma$, $E(A_i, B_j)$ is $\delta$-regular.
    \end{enumerate}
\end{theorem}
Szemer\'edi's proof shows that $K$ can be bounded above by an exponential tower with height a polynomial in $1/\delta$. Hopes of improving this enormous bound in general were quashed in 1997 when Gowers \cite{gowers1997} constructed graphs necessitating $K$ of this size.

However, various results have arisen since then that establish better bounds for $K$ in certain contexts, along with additional improvements on the regularity partition. Notably, in 2016, Fox, Pach, and Suk \cite{semialgebraicregularity} showed that when $E$ is semialgebraic, not only is $K$ upper bounded by a polynomial in $1/\delta$, but also (b) in Theorem \ref{szemeredi} can be replaced by the condition that for all $(i,j)\in [K]^2\setminus \Sigma$, $A_i\times B_j$ is $E$-homogeneous, a very strong form of regularity. In 2018, Chernikov and Starchenko \cite{regularitylemma} weakened the semialgebraicity assumption and showed that this holds if $E$ is definable in a distal structure, leading to the nomenclature \textit{distal regularity lemma}.

The results of Fox--Pach--Suk and Chernikov--Starchenko hold for relations of arbitrary arity (that is, for hypergraphs as well as graphs). We will state this result formally in an a priori roundabout way, by putting the spotlight on the degree of the polynomial in $1/\delta$ that upper bounds $K$ --- this will be important later on.

\begin{defn}\label{tupledefn}
    Let $M$ be a set, and let $E(x_1, ..., x_k)$ be a relation on $M$. Let $c_1, ..., c_k\in \R_{\geq 0}$, and write $\bar{c}:=(c_1, ..., c_k)$.
        
        Say that $\bar{c}$ is a \textit{distal regularity tuple (respectively, strong distal regularity tuple)} for $E$ if there is a \textit{coefficient} $\lambda>1$ satisfying the following: for all $\delta>0$ and finite sets $P_i\subseteq M^{x_i}$ with $n_i:=|P_i|$, there are partitions (respectively, equipartitions) $P_i=A^i_1\sqcup \cdots \sqcup A^i_{K_i}$ and an index set $\Sigma\subseteq [K_1]\times \cdots \times [K_k]$ of `bad cells' such that
    \begin{enumerate}[(a)]
        \item `Meagre bad cells': $\sum_{(j_1, ..., j_k)\in \Sigma}|A^1_{j_1}\times\cdots\times A^k_{j_k}|\leq \lambda\delta n_1\cdots n_k$;
        \item `Homogeneous good cells': for all $(j_1, ..., j_k)\in [K_1]\times \cdots\times [K_k]\setminus \Sigma$, $A^1_{j_1}\times\cdots\times A^k_{j_k}$ is $E$-homogeneous; and
        \item `Polynomially (in $\delta^{-1}$) many cells' For all $i\in [k]$, $K_i\leq \lambda\delta^{-c_i}$.
    \end{enumerate}
    
    Say that $E$ \textit{satisfies the distal regularity lemma} if there is a distal regularity tuple for $E$.
\end{defn}

It is immediate that strong distal regularity tuples are distal regularity tuples (with the same coefficient), and the following lemma establishes a converse.
\begin{lemma}\label{strongweak}
    Let $M$ be a set, and let $E(x_1, ..., x_k)$ be a relation on $M$. Let $c_1, ..., c_k\in \R_{\geq 0}$. Writing $\bar{c}:=(c_1, ..., c_k)$, suppose $\bar{c}$ is a distal regularity tuple for $E$ with coefficient $\lambda$. Then $(c_1+1, ..., c_k+1)$ is a strong distal regularity tuple for $E$ with coefficient $(k+2)\lambda$.
\end{lemma}
\begin{proof}
    Suppose $\bar{c}$ is a distal regularity tuple for $E$ with coefficient $\lambda$. Let $\delta\in (0,\lambda^{-1})$ and let $P_i\subseteq M^{x_i}$ be finite with $n_i:=|P_i|$. Fix partitions $P_i=A^i_1\sqcup \cdots \sqcup A^i_{K_i}$ and an index set of `bad cells' $\Sigma\subseteq [K_1]\times \cdots \times [K_k]$ as in the definition of $\bar{c}$ as a distal regularity tuple for $E$.
    
    For each $i\in [k]$, define an equipartition $P_i=B^i_1\sqcup\cdots\sqcup B^i_{L_i}$ as follows. Let $N_i:=\lceil\frac{1}{2}\delta^{c_i+1}n_i\rceil$. For each $j\in [K_i]$, partition $A^i_j$ into a maximal number of parts of size $N_i$ and a part $S^i_j$ of size less than $N_i$. This gives a new partition $P_i=S^i_1\sqcup\cdots\sqcup S^i_{K_i}\sqcup T^i_1\sqcup \cdots\sqcup T^i_{K'_i}$ where, for all $j'\in [K'_i]$, $|T^i_{j'}|=N_i$ and there is a unique $j\in [K_i]$ such that $T^i_{j'}\subseteq A^i_j$. Observe that $K'_i\leq n_i/N_i\leq 2\delta^{-(c_i+1)}$.
    
    By moving elements of $T^i_1, ..., T^i_{K'_i}$ to $S^i_1, ..., S^i_{K_i}$ as much as necessary, we obtain an equipartition $P_i=\bar{S}^i_1\sqcup\cdots\sqcup \bar{S}^i_{K_i}\sqcup \bar{T}^i_1\sqcup \cdots\sqcup \bar{T}^i_{K'_i}$ where each part has size at most $N_i$ and, for all $j'\in [K'_i]$, there is a unique $j\in [K_i]$ such that $\bar{T}^i_{j'}\subseteq A^i_j$. Rename $\bar{S}^i_1, ..., \bar{S}^i_{K_i}$ as $B^i_1, ..., B^i_{K_i}$ and $\bar{T}^i_1, ..., \bar{T}^i_{K'_i}$ as $B^i_{K_i+1}, ..., B^i_{L_i}$. Observe that $L_i=K_i+K'_i\leq (\lambda+2)\delta^{-(c_i+1)}$, and
    \[\sum_{j=1}^{K_i}|B^i_j|\leq \begin{cases}
        0&\text{if }N_i=1\\
        K_iN_i&\text{if }N_i>1
    \end{cases}\leq \lambda\delta n_i.\]
    
    For all $(j_1, ..., j_k)\in ([L_1]\setminus [K_1])\times\cdots\times([L_k]\setminus [K_k])$, there is a unique $(i_1, ..., i_k)\in [K_1]\timesdots [K_k]$ such that $B^1_{j_1}\timesdots B^k_{j_k}\subseteq A^1_{i_1}\timesdots A^k_{i_k}$; write $(i_1, ..., i_k)=\pi(j_1, ..., j_k)$. Now set
    \[\Lambda:=\bigcup_{i=1}^k [L_1]\times \cdots \times [L_{i-1}]\times [K_i]\times [L_{i+1}]\times\cdots \times [L_k]\cup\Lambda_0,\]
    where $\Lambda_0:=\{(j_1, ..., j_k)\in ([L_1]\setminus [K_1])\times\cdots\times([L_k]\setminus [K_k]): \pi(j_1, ..., j_k)\in\Sigma\}$. We claim that the partitions $P_i=B^i_1\sqcup\cdots \sqcup B^i_{L_i}$, together with the index set $\Lambda$ of bad cells, witness that $(c_1+1, ..., c_k+1)$ is a strong distal regularity tuple for $E$ with coefficient $(k+2)\lambda$.
    
    We already have that, for all $i\in [k]$, $L_i\leq (\lambda+2)\delta^{-(c_i+1)}\leq (k+2)\lambda\delta^{-(c_i+1)}$. Furthermore, for all $(j_1, ..., j_k)\in [L_1]\times \cdots \times [L_k]\setminus \Lambda\subseteq ([L_1]\setminus [K_1])\times\cdots\times([L_k]\setminus [K_k])$, we have $\pi(j_1, ..., j_k)\not\in\Sigma$; that is, $B^1_{j_1}\times\cdots\times B^k_{j_k}$ is contained in an $E$-homogeneous cell, so is itself $E$-homogeneous. Therefore, it remains to show that $\sum_{(j_1, ..., j_k)\in \Lambda}|B^1_{j_1}\times\cdots\times B^k_{j_k}|\leq (k+2)\lambda\delta n_1\cdots n_k$.
    
    Firstly, $\bigsqcup_{(j_1, ..., j_k)\in \Lambda_0}B^1_{j_1}\times\cdots\times B^k_{j_k}\subseteq \bigsqcup_{(i_1, ..., i_k)\in \Sigma}A^1_{i_1}\times\cdots\times A^k_{i_k}$, so the set on the left has size at most $\lambda\delta n_1\cdots n_k$. Now
    \begin{align*}
        \sum_{(j_1, ..., j_k)\in \Lambda}|B^1_{j_1}\times\cdots\times B^k_{j_k}|&= \sum_{i=1}^k n_1\cdots n_{i-1}n_{i+1}\cdots n_k \sum_{j=1}^{K_i}|B^i_j|+\sum_{(j_1, ..., j_k)\in \Lambda_0}|B^1_{j_1}\times\cdots\times B^k_{j_k}|\\
        &\leq \sum_{i=1}^k n_1\cdots n_{i-1}n_{i+1}\cdots n_k (\lambda \delta n_i)+\lambda\delta n_1\cdots n_k\\
        &=(k+1)\lambda\delta n_1\cdots n_k.\qedhere
    \end{align*}
\end{proof}
The results of Fox--Pach--Suk and Chernikov--Starchenko can now be stated as follows. We use $\R$ to denote the structure of the real ordered field.
\begin{theorem}[$\mathcal{M}=\R$: Fox--Pach--Suk, 2016 \cite{semialgebraicregularity}; $\mathcal{M}$ distal: Chernikov--Starchenko, 2018 \cite{regularitylemma}]\label{chernikovstarchenko}
    Let $\phi(x_1, ..., x_k; y)$ be a relation definable in a distal structure $\mathcal{M}$. Then there are $\bar{c}\in \R_{\geq 0}^k$ and $\lambda>1$ such that, for all $b\in M^y$, the relation $E(x_1, ..., x_k):=\phi(x_1, ..., x_k; b)$ on $M$ has distal regularity tuple $\bar{c}$ with coefficient $\lambda$.
    
    By Lemma \ref{strongweak}, `distal regularity tuple' can be replaced with `strong distal regularity tuple' in the statement above.
\end{theorem}

We will now forget about the context of $\mathcal{M}=\R$ or $\mathcal{M}$ as a distal structure (until Section \ref{sectiondistality}), and derive Zarankiewicz bounds for all relations with a (strong) distal regularity tuple, that is, all relations that satisfy the distal regularity lemma.

We finish this section by proving a preliminary Zarankiewicz bound for a relation $E(x_1, ..., x_k)$ with a distal regularity tuple, morally inducting on $k$.
\begin{lemma}\label{prelimbound}
        Let $E(x_1, ..., x_k)$ be a relation on a set $M$ with distal regularity tuple $\bar{c}=(c_1, ..., c_k)\in\R_{\geq 1}^k$ and coefficient $\lambda$. Suppose that, for all $i\in [k]$, $F_i: \N^{k-1}\to\R$ is a function satisfying the following.
        
        Let $u\in \N^+$, and let $a_1, ..., a_u\in M^{x_i}$ be distinct. For all $j\in [k]\setminus \{i\}$, let $P_j\subseteq M^{x_j}$ with $n_j:=|P_j|$. Then, if the $(k-1)$-graph $\bigcap_{e=1}^u E(P_1, ..., P_{i-1}, a_e, P_{i+1}, ..., P_k)$ is $K_{u, ..., u}$-free, then its size is $O_u(F_i(n_1, ..., n_{i-1}, n_{i+1}, ..., n_k))$.
\begin{enumerate}[(i)]
    \item Let $\gamma\geq 0$, and suppose that for all finite $P_i\subseteq M^{x_i}$ with $n_i:=|P_i|$, if $E(P_1, ..., P_k)$ is $K_{u, ..., u}$-free, then
    \[|E(P_1, ..., P_k)|\ll_{u,c_1,\lambda} n_1\cdots n_kn_1^{-\gamma}+n_1F_1(n_2, ..., n_k).\]
    Then the statement above holds with $\gamma$ replaced by $\frac{1}{1+c_1(1-\gamma)}$.
    \item For all finite $P_i\subseteq M^{x_i}$ with $n_i:=|P_i|$, if $E(P_1, ..., P_k)$ is $K_{u, ..., u}$-free, then for all $i\in [k]$ and $\varepsilon>0$,
    \[|E(P_1, ..., P_k)|\ll_{u,c_i,\lambda, \varepsilon} n_1\cdots n_kn_i^{-\frac{1}{c_i}+\varepsilon}+n_iF_i(n_1, ..., n_{i-1}, n_{i+1}, ..., n_k).\]
\end{enumerate}
\end{lemma}
\begin{proof}
    (i) Let $\delta=n_1^{-\frac{1}{1+c_1(1-\gamma)}}$. With this value of $\delta$, partition each $P_i=A^i_1\sqcup \cdots \sqcup A^i_{K_i}$ as in the definition of $\bar{c}$ as a distal regularity tuple for $E$, with $\Sigma\subseteq [K_1]\times \cdots \times [K_k]$ the index set of bad cells. Let $T:=\bigcup_{(j_1, ..., j_k)\in \Sigma} A^1_{j_1}\times\cdots\times A^k_{j_k}$. Without loss of generality, let $0\leq L\leq K_1$ be such that, for all $1\leq j\leq K_1$, $|A^1_j|\geq u$ if and only if $j>L$.

    Let $j>L$. Then $|A^1_j|\geq u$, so let $a_1, ..., a_u\in A^1_j$ be distinct. Then
    \[E(A^1_j, P_2, ..., P_k)\setminus T\subseteq A^1_j\times \bigcap_{e=1}^u E(a_e, P_2, ..., P_k).\]
    Since $E(P_1, ..., P_k)$ is $K_{u,...,u}$-free, the $(k-1)$-graph $\bigcap_{e=1}^u E(a_e, P_2, ..., P_k)$ is $K_{u,...,u}$-free, and so by assumption
    \[|E(A^1_j, P_2, ..., P_k)\setminus T|\ll_u |A^1_j|F_1(n_2, ..., n_k).\]
    
    Let $H_1:=\bigcup_{j=1}^L A^1_j$, so $|H_1|\leq Lu\leq K_1u\leq \lambda\delta^{-c_1}u$. By assumption, $|E(H_1, P_2, ..., P_k)|\ll_{u,c_1,\lambda} (\lambda\delta^{-c_1}u)^{1-\gamma}n_2\cdots n_k+n_1F_1(n_2, ..., n_k)\ll_{u,\lambda}\delta^{-c_1(1-\gamma)}n_2\cdots n_k+n_1F_1(n_2, ..., n_k)$. Thus,
    \begin{align*}
        &|E(P_1, ..., P_k)|\\
        &\leq |T|+|E(H_1, P_2, ..., P_k)|+\sum_{j=L+1}^{K_1} |E(A^1_j, P_2, ..., P_k)\setminus T|\\
        &\ll_{u,c_1,\lambda} \delta n_1\cdots n_k+\delta^{-c_1(1-\gamma)}n_2\cdots n_k+n_1F_1(n_2, ..., n_k)+\sum_{j=L+1}^{K_1} |A^1_j|F_1(n_2, ..., n_k)\\
        &\leq 2n_1\cdots n_kn_1^{-\frac{1}{1+c_1(1-\gamma)}}+2n_1F_1(n_2, ..., n_k).
    \end{align*}

    (ii) By symmetry, we may assume that $i=1$. Let $f:[0,\frac{1}{c_1}]\to [\frac{1}{c_1+1},\frac{1}{c_1}]$ be given by $\gamma \mapsto \frac{1}{1+c_1(1-\gamma)}$. The statement in (i) holds for $\gamma=0$, so it suffices to show that $f^n(0)\to \frac{1}{c_1}$ as $n\to\infty$. Note that for all $\gamma\in [0,\frac{1}{c_1}]$ we have $\gamma\leq f(\gamma)$: indeed, $(c_1\gamma-1)(\gamma-1)\geq 0$, which rearranges to $\gamma(1+c_1(1-\gamma))\leq 1$. Thus, $(f^n(0))_n$ is an increasing sequence in $[\frac{1}{c_1+1},\frac{1}{c_1}]$, and so it converges to some limit $L\in [\frac{1}{c_1+1},\frac{1}{c_1}]$. But then $L=\frac{1}{1+c_1(1-L)}$, which rearranges to $(c_1L-1)(L-1)=0$, and so $L=\frac{1}{c_1}$ since $c_1\geq 1$.
\end{proof}
\begin{remark}\label{prelimboundremark}
    In Lemma \ref{prelimbound}, when $k=2$, $F_i$ can be chosen to be the constant $1$-valued function. Indeed, if a $1$-graph is $K_u$-free, then its size is at most $u-1=O_u(1)$.
\end{remark}
\begin{remark}
    After the preparation of this paper, we were made aware of a Tur\'an-type argument in \cite[Corollary 5.1]{hans} which allows one to remove the $\varepsilon$ from the bound in Lemma \ref{prelimbound}(ii) as long as $c_i>1$. Even so, bootstrapping this infinitesimally improved bound via our methods does not allow us to remove the $\varepsilon$ in our main theorem (Theorem \ref{mainabstractbound}) or its binary counterpart (Theorem \ref{graphs}), so we retain the statement and proof of Lemma \ref{prelimbound} as written to provide a different perspective and proof method.
\end{remark}
\section{Binary relations}\label{sectionbinary}
We will first consider binary relations, for two reasons. Firstly, for binary relations, our main theorem holds under a weaker assumption --- namely, $\bar{c}$ is only required to be a distal regularity tuple, not a strong distal regularity tuple. Secondly, the exposition is much cleaner for binary relations, and so will hopefully illuminate the proof strategy for arbitrary relations.
\begin{theorem}\label{graphs}
    Let $E(x, y)$ be a relation on a set $M$, with distal regularity tuple $\bar{c}=(c_1, c_2)\in\R_{\geq 1}^2$ and coefficient $\lambda$. Then, for all finite $P\subseteq M^x$ and $Q\subseteq M^y$ with $m:=|P|$ and $n:=|Q|$, if $E(P,Q)$ is $K_{u,u}$-free, then for all $\varepsilon>0$ we have
    \[|E(P,Q)|\ll_{u, \bar{c}, \lambda, \varepsilon}m^{\frac{c_2(c_1-1)}{c_1c_2-1}+\varepsilon}n^{\frac{c_1(c_2-1)}{c_1c_2-1}+\varepsilon}+m+n,\]
    where $\frac{c_2-1}{c_1c_2-1}, \frac{c_1-1}{c_1c_2-1}:= \lim_{\delta\to 0}\frac{(1+\delta)-1}{(1+\delta)^2-1}=\frac{1}{2}$ if $c_1=c_2=1$.
\end{theorem}
\begin{proof}
    We will show that, for all $\varepsilon>0$, there are constants $\alpha=\alpha(u,\bar{c},\lambda,\varepsilon)$ and $\beta=\beta(u,\bar{c},\lambda,\varepsilon)$ such that, for all finite $P\subseteq M^x$ and $Q\subseteq M^y$ with $m:=|P|$ and $n:=|Q|$, if $E(P, Q)$ is $K_{u, u}$-free, then
    \begin{equation}\label{binaryeqnalphabeta}
        |E(P, Q)|\leq \alpha m^{\frac{c_2(c_1-1)}{c_1c_2-1}+\varepsilon}n^{\frac{c_1(c_2-1)}{c_1c_2-1}+\varepsilon}+\beta(m+n).
    \end{equation}
    The dependency between constants will be as follows:
    \begin{enumerate}[(i)]
        \item $\delta$ is sufficiently small in terms of $\bar{c}$, $\lambda$, and $\varepsilon$;
        \item $m_0$ is sufficiently large in terms of $u$, $\bar{c}$, $\lambda$, $\varepsilon$, and $\delta$;
        \item $\beta$ is sufficiently large in terms of $m_0$, $u$, $\bar{c}$, $\lambda$, and $\varepsilon$; and
        \item $\alpha$ is sufficiently large in terms of $m_0$, $\beta$, $u$, $\bar{c}$, $\lambda$, and $\delta$.
    \end{enumerate}
    
    By Lemma \ref{prelimbound}(ii) and Remark \ref{prelimboundremark}, $\abs{E(P, Q)}\ll_{u,\bar{c},\lambda,\varepsilon} m^{1-\frac{1-\nu}{c_1}}n+m$ and $\abs{E(P, Q)}\ll_{u,\bar{c},\lambda,\varepsilon} mn^{1-\frac{1-\nu}{c_2}}+n$, where $\nu\in (0,\frac{1}{2})$ is chosen such that
    \[\left(\frac{1}{1-\nu}-1\right)\max\left(\frac{c_2(c_1-1)}{c_1c_2-1},\frac{c_1(c_2-1)}{c_1c_2-1}\right)\leq \varepsilon.\]
    Note then in particular that $\frac{c_1-1}{c_1c_2-1}\frac{c_2}{1-\nu}=\frac{1}{1-\nu}\frac{c_2(c_1-1)}{c_1c_2-1}\leq \frac{c_2(c_1-1)}{c_1c_2-1}+\varepsilon$, and similarly $\frac{c_2-1}{c_1c_2-1}\frac{c_1}{1-\nu}=\frac{1}{1-\nu}\frac{c_1(c_2-1)}{c_1c_2-1}\leq \frac{c_1(c_2-1)}{c_1c_2-1}+\varepsilon$.
    
    If $m\leq n^{\frac{1-\nu}{c_2}}$ and $\beta$ is sufficiently large in terms of $u$, $\bar{c}$, $\lambda$, and $\varepsilon$, then $\abs{E(P, Q)}\leq \beta n$ since $\abs{E(P, Q)}\ll_{u,\bar{c},\lambda,\varepsilon} mn^{1-\frac{1-\nu}{c_2}}+n$. Therefore, for the rest of the proof we assume that $n<m^{\frac{c_2}{1-\nu}}$, which implies
    \begin{equation}\label{nbound}
        n=n^{\frac{c_1-1}{c_1c_2-1}}n^{\frac{c_1(c_2-1)}{c_1c_2-1}}\leq m^{\frac{c_2(c_1-1)}{c_1c_2-1}+\varepsilon}n^{\frac{c_1(c_2-1)}{c_1c_2-1}}\leq m^{\frac{c_2(c_1-1)}{c_1c_2-1}+\varepsilon}n^{\frac{c_1(c_2-1)}{c_1c_2-1}+\varepsilon}.
    \end{equation}
    Similarly, for the rest of the proof we assume that $m<n^{\frac{c_1}{1-\nu}}$, which implies
    \begin{equation}\label{mbound}
        m=m^{\frac{c_2-1}{c_2c_1-1}}m^{\frac{c_2(c_1-1)}{c_2c_1-1}}\leq n^{\frac{c_1(c_2-1)}{c_2c_1-1}+\varepsilon}m^{\frac{c_2(c_1-1)}{c_2c_1-1}}\leq m^{\frac{c_2(c_1-1)}{c_1c_2-1}+\varepsilon}n^{\frac{c_1(c_2-1)}{c_1c_2-1}+\varepsilon}.
    \end{equation}
    
    Let $\alpha=\alpha(u,\bar{c},\lambda,\varepsilon)$ and $\beta=\beta(u,\bar{c},\lambda,\varepsilon)$ be sufficiently large, to be chosen later. We show by induction on $m+n$ that (\ref{binaryeqnalphabeta}) holds.

    Let $m_0:=m_0(u,\bar{c},\lambda,\varepsilon,\delta)$ such that
    \begin{equation}\label{defnm_0diamond}
        m_0^{\frac{c_2(c_1-1)}{c_1c_2-1}+\varepsilon}>4(\lambda\delta^{-c_1}u)^{\frac{c_2(c_1-1)}{c_1c_2-1}+\varepsilon}.
    \end{equation}
    If $m+n<m_0$, then (\ref{binaryeqnalphabeta}) holds by choosing values for $\alpha$ and $\beta$ that are sufficiently large in terms of $m_0$. Thus, henceforth assume that $m+n\geq m_0$, and suppose that (\ref{binaryeqnalphabeta}) holds when $|P|+|Q|<m+n$. If $m<m_0$, then $|E(P,Q)|<m_0n\leq \beta n$ assuming $\beta\geq m_0$, so henceforth suppose $m\geq m_0$.

    For $\delta:=\delta(\bar{c}, \lambda,\varepsilon)<1$ to be chosen later, partition $P=A_1\sqcup \cdots \sqcup A_{K_1}$ and $Q=B_1\sqcup \cdots \sqcup B_{K_2}$ as in the definition of $\bar{c}$ as a distal regularity tuple, with $\Sigma\subseteq [K_1]\times [K_2]$ the index set of bad cells. By refining the partition and replacing $\lambda$ with $\lambda+2$ if necessary, we can assume that $\abs{A_i}\leq \delta^{c_1}m$ for all $i\in [K_1]$.
    
    For $i\in [K_1]$, let $\Sigma_i:=\{j\in [K_2]: (i,j)\in\Sigma\}$. Without loss of generality, let $0\leq L\leq L'\leq K_1$ be such that:
    \begin{enumerate}[(i)]
        \item For all $i\in [K_1]$, $\abs{A_i}\geq u$ if and only if $i>L$; and
        \item For all $i\in [K_1]\setminus [L]$, $\sum_{j\in \Sigma_i}\abs{B_j}\leq \delta^{1-\varepsilon}n$ if and only if $i>L'$.
    \end{enumerate}
    Partition $P$ into $H_1:=\bigcup_{i=1}^L A_i$, $H_2:=\bigcup_{i=L+1}^{L'} A_i$, and $H_3:=\bigcup_{i=L'+1}^{K_1} A_i$. We will bound $|E(P,Q)|$ by bounding $\abs{E(H_1, Q)}$, $\abs{E(H_2,Q)}$, and $\abs{E(H_3, Q)}$.
    
    Consider $E(H_1, Q)$. Note that $\abs{H_1}\leq Lu\leq K_1u\leq \lambda\delta^{-c_1}u$. Choosing $m_0>\lambda\delta^{-c_1}u$, we have $m\geq m_0>\lambda\delta^{-c_1}u\geq \abs{H_1}$. By the induction hypothesis,
    \begin{align*}
        \abs{E(H_1, Q)}&\leq \alpha (\lambda\delta^{-c_1}u)^{\frac{c_2(c_1-1)}{c_1c_2-1}+\varepsilon}n^{\frac{c_1(c_2-1)}{c_1c_2-1}+\varepsilon}+\beta(\lambda\delta^{-c_1}u+n)\\
        &\leq \frac{\alpha}{4}m^{\frac{c_2(c_1-1)}{c_1c_2-1}+\varepsilon}n^{\frac{c_1(c_2-1)}{c_1c_2-1}+\varepsilon}+\beta(m+n),
    \end{align*}
    recalling that $m\geq m_0$ and applying (\ref{defnm_0diamond}).

    Next, consider $E(H_2, Q)$. By definition, $\bigcup_{i=L+1}^{L'}\bigcup_{j\in \Sigma_i}A_i\times B_j\subseteq \bigcup_{(i,j)\in\Sigma}A_i\times B_j$. The set on the right has size at most $\lambda \delta mn$, and for all $L+1\leq i\leq L'$ we have $\sum_{j\in \Sigma_i}\abs{B_j}> \delta^{1-\varepsilon}n$. Thus,
    \[\abs{H_2}=\abs{\bigcup_{i=L+1}^{L'} A_i}< \frac{\lambda\delta mn}{\delta^{1-\varepsilon}n}=\lambda \delta^\varepsilon m.\]
    In particular, assuming $\delta$ is sufficiently small in terms of $\lambda$ and $\varepsilon$, we have $\abs{H_2}<m$, so by the induction hypothesis,
    \begin{align*}
        \abs{E(H_2, Q)}&< \alpha (\lambda\delta^{\varepsilon}m)^{\frac{c_2(c_1-1)}{c_1c_2-1}+\varepsilon}n^{\frac{c_1(c_2-1)}{c_1c_2-1}+\varepsilon}+\beta(\lambda\delta^{\varepsilon}m+n)\\
        &\leq \frac{\alpha}{4} m^{\frac{c_2(c_1-1)}{c_1c_2-1}+\varepsilon}n^{\frac{c_1(c_2-1)}{c_1c_2-1}+\varepsilon}+\beta(m+n)
    \end{align*}
    for $\delta$ sufficiently small in terms of $\lambda$ and $\varepsilon$.

    Next, consider $E(H_3, Q)$. We will bound its size by partitioning it into two:
    \[E(H_3, Q)=\left(\bigcup_{i=L'+1}^{K_1}\bigcup_{j\in [K_2]\setminus \Sigma_i} E(A_i, B_j)\right)\sqcup \left(\bigcup_{i=L'+1}^{K_1}\bigcup_{j\in \Sigma_i} E(A_i, B_j)\right).\]
    
    Fix $L'+1\leq i\leq K_1$. For $j\in [K_2]\setminus \Sigma_i$, $E(A_i, B_j)=A_i\times B_j$ or $\emptyset$. Since $\abs{A_i}\geq u$ and $E(A_i, Q)$ is $K_{u,u}$-free, we thus have that $\abs{E(A_i, \bigcup_{j\in [K_2]\setminus \Sigma_i}B_j)}\leq (u-1)\abs{A_i}$. Hence,
    \[\sum_{i=L'+1}^{K_1}\sum_{j\in [K_2]\setminus \Sigma_i}\abs{E(A_i, B_j)}\leq (u-1)m.\]

    Now, $\sum_{j\in \Sigma_i}\abs{B_j}\leq \delta^{1-\varepsilon}n$ by definition. Recall also that $\abs{A_i}\leq \delta^{c_1}m$. In particular, $\abs{A_i}<m$, so by the induction hypothesis,
    \begin{align*}
        \abs{E\left(A_i, \bigcup_{j\in \Sigma_i}B_j\right)}&\leq \alpha (\delta^{c_1} m)^{\frac{c_2(c_1-1)}{c_1c_2-1}+\varepsilon}(\delta^{1-\varepsilon}n)^{\frac{c_1(c_2-1)}{c_1c_2-1}+\varepsilon}+\beta(\delta^{c_1}m+\delta^{1-\varepsilon}n)\\
        &= \alpha \delta^{c_1+\varepsilon(c_1-\frac{c_1(c_2-1)}{c_1c_2-1}+1-\varepsilon)} m^{\frac{c_2(c_1-1)}{c_1c_2-1}+\varepsilon}n^{\frac{c_1(c_2-1)}{c_1c_2-1}+\varepsilon}+\beta(\delta^{c_1}m+\delta^{1-\varepsilon}n)\\
        &\leq \frac{\alpha}{5} \delta^{c_1} m^{\frac{c_2(c_1-1)}{c_1c_2-1}+\varepsilon}n^{\frac{c_1(c_2-1)}{c_1c_2-1}+\varepsilon}+\beta(m+n)
    \end{align*}
    for $\delta$ sufficiently small in terms of $\bar{c}$ and $\varepsilon$. Thus,
    \begin{align*}
        \sum_{i=L'+1}^{K_1}\sum_{j\in \Sigma_i}\abs{E(A_i, B_j)}&\leq \lambda\delta^{-c_1}\frac{\alpha}{5} \delta^{c_1} m^{\frac{c_2(c_1-1)}{c_1c_2-1}+\varepsilon}n^{\frac{c_1(c_2-1)}{c_1c_2-1}+\varepsilon}+\lambda\delta^{-c_1}\beta(m+n)\\
        &\leq \frac{\alpha}{4} m^{\frac{c_2(c_1-1)}{c_1c_2-1}+\varepsilon}n^{\frac{c_1(c_2-1)}{c_1c_2-1}+\varepsilon}+\lambda\delta^{-c_1}\beta(m+n)
    \end{align*}
    for $\alpha$ sufficiently large in terms of $\lambda$.

    Putting all this together,
    \begin{align*}
        |E(P,Q)|&= \abs{E(H_1, Q)}+\abs{E(H_2, Q)} + \sum_{i=L'+1}^{K_1}\sum_{j\in [K_2]\setminus \Sigma_i}\abs{E(A_i, B_j)}+\sum_{i=L'+1}^{K_1}\sum_{j\in \Sigma_i}\abs{E(A_i, B_j)}\\
        &\leq \frac{3\alpha}{4}m^{\frac{c_2(c_1-1)}{c_1c_2-1}+\varepsilon}n^{\frac{c_1(c_2-1)}{c_1c_2-1}+\varepsilon}+(2\beta+\lambda\delta^{-c_1}\beta+u-1)(m+n)\\
        &\leq \alpha m^{\frac{c_2(c_1-1)}{c_1c_2-1}+\varepsilon}n^{\frac{c_1(c_2-1)}{c_1c_2-1}+\varepsilon},
    \end{align*}
    where the last inequality was obtained from (\ref{nbound}) and (\ref{mbound}), choosing $\alpha$ to be sufficiently large in terms of $\beta$, $u$, $\bar{c}$, $\lambda$, and $\delta$. Thus, (\ref{binaryeqnalphabeta}) holds as claimed.
\end{proof}
\begin{remark}\label{binaryepsilonremark}
    It is straightforward to observe that the bound in Theorem \ref{graphs} can be infinitesimally improved to, say, $m^{\frac{c_2(c_1-1)}{c_1c_2-1}+\varepsilon}n^{\frac{c_1(c_2-1)}{c_1c_2-1}}+m+n$. Indeed, if $m\leq n^{\frac{1}{2c_1}}$ then $|E(P,Q)|\ll_{u,\bar{c},\lambda} n$ by Lemma \ref{prelimbound}. Assuming therefore, without loss of generality, that $n<m^{2c_1}$, by Theorem \ref{graphs},
    \[|E(P,Q)|\ll_{u, \bar{c}, \lambda, \varepsilon}m^{\frac{c_2(c_1-1)}{c_1c_2-1}+\varepsilon}n^{\frac{c_1(c_2-1)}{c_1c_2-1}}m^{2c_1\varepsilon}+m+n,\]
    and so
    \[|E(P,Q)|\ll_{u, \bar{c}, \lambda, \varepsilon}m^{\frac{c_2(c_1-1)}{c_1c_2-1}+\varepsilon}n^{\frac{c_1(c_2-1)}{c_1c_2-1}}+m+n.\]
\end{remark}
    
\section{The general case}\label{sectionmain}
We now proceed with the proof of the main theorem for an arbitrary relation.
\begin{lemma}\label{tupleinduction}
    Let $E(x_1, ..., x_k)$ be a relation on a set $M$, with strong distal regularity tuple $\bar{c}=(c_1, ..., c_k)\in\R_{\geq 1}^k$ and coefficient $\lambda$. For all $u\in\N^+$ and distinct $a_1, ..., a_u\in M^{x_1}$, the relation
    \[R(x_2, ..., x_k):=\bigwedge_{e=1}^u E(a_e, x_2, ..., x_k)\]
    has strong distal regularity tuple $(c_2, ..., c_k)$ with coefficient $u\lambda$.
\end{lemma}
\begin{proof}
    Let $P_1:=\{a_1, ..., a_u\}$ for $a_1, ..., a_u\in M^{x_1}$ distinct, and for $2\leq i\leq k$, let $P_i\subseteq M^{x_i}$ be finite with $n_i:=|P_i|$.

    Let $\delta\in (0,u^{-1})$. With this value of $\delta$, obtain equipartitions $P_i=A^i_1\sqcup \cdots\sqcup A^i_{K_i}$ as in the definition of $\bar{c}$ as a strong distal regularity tuple for $E$, and let $\Sigma\subseteq [K_1]\times\cdots \times [K_k]$ be the index set of bad cells. Since $1\geq \delta^{c_1}|P_1|$, we can assume without loss of generality that the partition of $P_1$ is a partition into singletons, and $A^1_j=\{a_j\}$ for all $j\in [u]$.
    
    Henceforth, for readability, a tuple $(j_2, ..., j_k)$ is understood to be taken from $[K_2]\times \cdots\times [K_k]$. Let
    \[\Sigma':=\{(j_2, ..., j_k): \exists j_1\in [u]\;(j_1, ..., j_k)\in\Sigma\}.\]
    We claim that the equipartitions $P_i=A^i_1\sqcup \cdots\sqcup A^i_{K_i}$ (for $2\leq i\leq k$) and the index set $\Sigma'$ of bad cells are such that
    \begin{enumerate}[(i)]
        \item $\sum_{(j_2, ..., j_k)\in \Sigma'}|A^2_{j_2}\times\cdots\times A^k_{j_k}|\leq u\lambda\delta n_2\cdots n_k$;
        \item For all $(j_2, ..., j_k)\not\in\Sigma'$, $A^2_{j_2}\times\cdots \times A^k_{j_k}$ is $R$-homogeneous;
        \item $K_i\leq \lambda \delta^{-c_i}$ for all $2\leq i\leq k$.
    \end{enumerate}

    To see that (i) holds, observe that
    \[\sum_{(j_2, ..., j_k)\in \Sigma'}|A^2_{j_2}\times\cdots\times A^k_{j_k}|\leq \sum_{(j_1, ..., j_k)\in \Sigma}|A^1_{j_1}\times\cdots\times A^k_{j_k}|\leq \lambda\delta un_2\cdots n_k.\]
    
    To see that (ii) holds, let $(j_2, ..., j_k)\not\in\Sigma'$ and $(b_2, ..., b_k)\in A^2_{j_2}\times\cdots\times A^k_{j_k}$. Then
    \[R(b_2, ..., b_k)\Leftrightarrow \bigwedge_{e=1}^u E(a_e, b_2, ..., b_k)\Leftrightarrow \bigwedge_{e=1}^u E(A^1_e, A^2_{j_2}, ..., A^k_{j_k})=A^1_e\times A^2_{j_2}\times\cdots\times A^k_{j_k},\]
    where the last equivalence follows from the fact that, for all $e\in [u]$, $A^1_e=\{a_e\}$ and $(e, j_2, ..., j_k)\not\in\Sigma$. Thus, $A^2_{j_2}\times\cdots \times A^k_{j_k}$ is $R$-homogeneous.

    Finally, (iii) holds by the choice of our original partition. Thus, $(c_2, ..., c_k)$ is a strong distal regularity tuple for $R(x_2, ..., x_k)$ with coefficient $u\lambda$.
\end{proof}
    The following functions appeared in \cite{do}.
\begin{defn}\label{Fdefn}
    For $\bar{c}=(c_1, ..., c_k)\in\R_{\geq 1}^k$, let $E_{\bar{c}}: \R_{\geq 0}^k\to\R$ be the function sending $\bar{n}=(n_1, ..., n_k)\in\R_{\geq 0}^k$ to $E_{\bar{c}}(\bar{n}):=\prod_{i=1}^k n_i^{\gamma_i(\bar{c})}$, where
    \[\gamma_i(\bar{c}):=1-\frac{\frac{1}{c_i-1}}{k-1+\sum_{j=1}^k\frac{1}{c_j-1}}.\]
    Note that, when $k=1$, $E_{\bar{c}}$ is the constant 1-valued function. For $\varepsilon\in\R_{>0}$, if $k\geq 2$ then let $F^\varepsilon_{\bar{c}}:\R_{\geq 0}^k\to\R$ be the function sending $\bar{n}=(n_1, ..., n_k)\in\R_{\geq 0}^k$ to
    \[F^\varepsilon_{\bar{c}}(\bar{n}):=\sum_{I\subseteq [k], |I|\geq 2}E_{\bar{c}_I}(\bar{n}_I)\prod_{i\in I}n_i^\varepsilon\prod_{i\in [k]\setminus I}n_i+\sum_{j=1}^k \prod_{i\in [k]\setminus \{j\}}n_i,\]
    and if $k=1$ then let $F^\varepsilon_{\bar{c}}:\R_{\geq 0}\to\R$ be the constant $1$-valued function. (Recall the notation of $\bar{c}_I$, $\bar{n}_I$ from Sub-subsection \ref{subsubsecindexing}.)
\end{defn}
As written, the exponents in $E_{\bar{c}}(\bar{n})$ are not well-defined when $c_j=1$ for some $j$. In this case, we circumvent this problem by declaring, for all $i\in [k]$,
\[\gamma_i(\bar{c}):=1-\lim_{\delta\to 0}\frac{\frac{1}{c_i+\delta-1}}{k-1+\sum_{j=1}^k\frac{1}{c_j+\delta-1}}=1-\frac{\mathbbm{1}(c_i=1)}{|\{j\in [k]: c_j=1\}|}.\]

Note that, when $k=2$,
\begin{align*}
    F^\varepsilon_{\bar{c}}(m,n)&=m^{1-\frac{\frac{1}{c_1-1}}{1+\frac{1}{c_1-1}+\frac{1}{c_2-1}}+\varepsilon}n^{1-\frac{\frac{1}{c_2-1}}{1+\frac{1}{c_1-1}+\frac{1}{c_2-1}}+\varepsilon}+m+n\\
    &=m^{1-\frac{c_2-1}{c_1c_2-1}+\varepsilon}n^{1-\frac{c_1-1}{c_1c_2-1}+\varepsilon}+m+n\\
    &=m^{\frac{c_2(c_1-1)}{c_1c_2-1}+\varepsilon}n^{\frac{c_1(c_2-1)}{c_1c_2-1}+\varepsilon}+m+n,
\end{align*}
so $F^\varepsilon_{\bar{c}}(m,n)$ is the bound appearing in Theorem \ref{graphs}.
\begin{remark}\label{inductiveremark}
    It is straightforward to observe that, when $k\geq 2$,
    \[kF^\varepsilon_{\bar{c}}(\bar{n})\geq E_{\bar{c}}(\bar{n})\prod_{i=1}^k n_i^\varepsilon+\sum_{i=1}^k n_i F^\varepsilon_{\bar{c}_{\neq i}}(\bar{n}_{\neq i}).\]
    The following lemma says that (in most cases) the dominant term in $F^\varepsilon_{\bar{c}}(\bar{n})$ is $E_{\bar{c}}(\bar{n})\prod_{i=1}^k n_i^\varepsilon$.
\end{remark}
\begin{lemma}\label{dumplemma}
    Let $k\geq 2$, $\bar{c}=(c_1, ..., c_k)\in\R_{\geq 1}^k$, $\varepsilon>0$, and $\bar{n}=(n_1, ..., n_k)\in\R_{\geq 0}^k$. Suppose that, for all $i\in [k]$,
    \[n_1\cdots n_kn_i^{-1/c_i+\varepsilon}\geq n_iF^\varepsilon_{\bar{c}_{\neq i}}(\bar{n}_{\neq i}).\]
    Then $F^\varepsilon_{\bar{c}}(\bar{n})\ll_{\bar{c},\varepsilon} E_{\bar{c}}(\bar{n})\prod_{i=1}^k n_i^\varepsilon$, and so $n_i F^\varepsilon_{\bar{c}_{\neq i}}(\bar{n}_{\neq i})\ll_{\bar{c},\varepsilon} E_{\bar{c}}(\bar{n})\prod_{i=1}^k n_i^\varepsilon$ for all $i\in [k]$ by Remark \ref{inductiveremark}.
\end{lemma}
\begin{proof}
    Our proof mimics, in part, the proof of \cite[Lemma 2.10]{do}. To show that $F^\varepsilon_{\bar{c}}(\bar{n})\ll_{\bar{c},\varepsilon} E_{\bar{c}}(\bar{n})\prod_{i=1}^k n_i^\varepsilon$, it suffices to show that, for all $\emptyset\neq I\subseteq [k]$,
    \begin{equation}\label{dominanteqn}
        E_{\bar{c}}(\bar{n})\gg_{\bar{c},\varepsilon} E_{\bar{c}_I}(\bar{n}_I)\prod_{i\not\in I}n_i^{1-\varepsilon}.
    \end{equation}
    We prove this by downward induction on $|I|\in [k]$ via the following claim.
    \renewcommand\qedsymbol{$\dashv$}
    \begin{claim}\label{dominantclaim}
        Let $J\subseteq [k]$ with $|J|\geq 2$. Let $j\in J$, and write $I:=J\setminus \{j\}$. For all $\varepsilon>0$, if $n_j^{-1/c_j+\varepsilon}\prod_{i\in J}n_i\geq n_jE_{\bar{c}_I}(\bar{n}_I)$ then $E_{\bar{c}_J}(\bar{n}_J)\geq n_j^{1-\varepsilon} E_{\bar{c}_I}(\bar{n}_I)$.
    \end{claim}
    \begin{proof}[Proof of Claim]
        Throughout this argument, if necessary, replace each $c_i$ with $c_i+\delta$ and take the limit as $\delta\to 0$.
        
        Let $\varepsilon>0$, and suppose $n_j^{-1/c_j+\varepsilon}\prod_{i\in J}n_i\geq n_jE_{\bar{c}_I}(\bar{n}_I)$. Then
        \[\prod_{i\in I}n_i^{\frac{\frac{1}{c_i-1}}{|J|-2+\sum_{l\in I}\frac{1}{c_l-1}}}\geq n_j^{\frac{1}{c_j}-\varepsilon}.\]
        The $i^{\text{th}}$ exponent on the left equals $\frac{c_j-1}{c_j}\frac{1}{c_i-1}\left(\frac{|J|-1+\sum_{l\in J} \frac{1}{c_l-1}}{|J|-2+\sum_{l\in I}\frac{1}{c_l-1}}-1\right)$, and so
        \[\prod_{i\in I}n_i^{\frac{1}{c_i-1}\left(\frac{1}{|J|-2+\sum_{l\in I}\frac{1}{c_l-1}}-\frac{1}{|J|-1+\sum_{l\in J}\frac{1}{c_l-1}}\right)}\geq n_j^{\frac{\frac{1}{c_j-1}}{|J|-1+\sum_{l\in J}\frac{1}{c_l-1}}-\nu\varepsilon}\]
        for $\nu:=\frac{c_j/(c_j-1)}{|J|-1+\sum_{l\in J} \frac{1}{c_l-1}}\in [0,1]$. Rearranging, $E_{\bar{c}_J}(\bar{n}_J)\geq n_j^{1-\nu\varepsilon} E_{\bar{c}_I}(\bar{n}_I)\geq n_j^{1-\varepsilon} E_{\bar{c}_I}(\bar{n}_I)$.
    \end{proof}
    \renewcommand\qedsymbol{$\square$}
    We now prove (\ref{dominanteqn}) by downward induction on $\abs{I}\in [k]$. Since $k$ is finite, we may update the implied constant in each step of the induction. When $\abs{I}=k$, we have $I=[k]$ so (\ref{dominanteqn}) holds trivially. Now suppose $\abs{I}<k$, so we may fix $j\not\in I$; write $J:=I\cup\{j\}$. Since $|I|\geq 1$, we have $|J|\geq 2$. By the induction hypothesis, $E_{\bar{c}}(\bar{n})\gg_{\bar{c},\varepsilon} E_{\bar{c}_J}(\bar{n}_J)\prod_{i\not\in J}n_i^{1-\varepsilon}$.

    By assumption, $n_1\cdots n_kn_j^{-1/c_j+\varepsilon}\geq n_jF^\varepsilon_{\bar{c}_{\neq j}}(\bar{n}_{\neq j})\geq E_{\bar{c}_I}(\bar{n}_I)\prod_{i\not\in I}n_i$, which rearranges to $n_j^{-1/c_j+\varepsilon}\prod_{i\in J}n_i\geq n_jE_{\bar{c}_I}(\bar{n}_I)$, and hence $E_{\bar{c}_J}(\bar{n}_J)\geq n_j^{1-\varepsilon} E_{\bar{c}_I}(\bar{n}_I)$ by the Claim. Thus,
    \[E_{\bar{c}}(\bar{n})\gg_{\bar{c},\varepsilon} E_{\bar{c}_J}(\bar{n}_J)\prod_{i\not\in J}n_i^{1-\varepsilon}\geq E_{\bar{c}_I}(\bar{n}_I)\prod_{i\not\in I}n_i^{1-\varepsilon}\]
    as required.
\end{proof}
\begin{theorem}\label{mainabstractbound}
    Let $E(x_1, ..., x_k)$ be a relation on a set $M$, with strong distal regularity tuple $\bar{c}=(c_1, ..., c_k)\in\R_{\geq 1}^k$ and coefficient $\lambda$. For all finite $P_i\subseteq M^{x_i}$ with $n_i:=|P_i|$, if $E(P_1, ..., P_k)$ is $K_{u, ..., u}$-free, then for all $\varepsilon>0$,
    \[|E(P_1, ..., P_k)|\ll_{u,\bar{c},\lambda,\varepsilon}F^\varepsilon_{\bar{c}}(n_1, ..., n_k).\]
\end{theorem}
\begin{proof}
    We will do a double induction: first on $k$, and then on $n_1+\cdots+n_k$. When $k=1$ this is trivial. Let $k\geq 2$, and suppose for all $l<k$ that the statement holds. Writing $\gamma_j:=\gamma_j(\bar{c})$ for $j\in [k]$, there is some $j\in [k]$ such that $\gamma_j<1$, so, permuting $x_1, ..., x_k$ if necessary, we may assume that $\gamma_1<1$. Let $\varepsilon>0$. We will show that there are $\alpha=\alpha(u,\bar{c},\lambda,\varepsilon)$ and $\beta=\beta(u,\bar{c},\lambda,\varepsilon)$ such that, for all finite $P_i\subseteq M^{x_i}$ with $n_i:=|P_i|$, if $E(P_1, ..., P_k)$ is $K_{u, ..., u}$-free, then
    \begin{equation}\label{mainboundalphabeta}
        |E(P_1, ..., P_k)|\leq \alpha E_{\bar{c}}(\bar{n})\prod_{i=1}^k n_i^\varepsilon+\beta \sum_{i=1}^k n_i F^\varepsilon_{\bar{c}_{\neq i}}(\bar{n}_{\neq i}).
    \end{equation}
    By Remark \ref{inductiveremark}, this is sufficient for the inductive step. The dependency between constants will be as follows:
    \begin{enumerate}[(i)]
        \item $\tau$ is sufficiently large in terms of $\bar{c}$ and $\varepsilon$;
        \item $\delta$ is sufficiently small in terms of $\bar{c}$, $\lambda$, and $\varepsilon$;
        \item $m_0$ is sufficiently large in terms of $u$, $\bar{c}$, $\lambda$, $\varepsilon$, and $\delta$;
        \item $\beta$ is sufficiently large in terms of $m_0$, $u$, $\bar{c}$, $\lambda$, and $\varepsilon$; and
        \item $\alpha$ is sufficiently large in terms of $m_0$, $\beta$, $\bar{c}$, $\lambda$, $\delta$, and $\tau$.
    \end{enumerate}

    Suppose there is $i\in [k]$ such that $n_1\cdots n_kn_i^{-1/c_i+\varepsilon}< n_iF^\varepsilon_{\bar{c}_{\neq i}}(\bar{n}_{\neq i})$. Then, by Lemma \ref{prelimbound}(ii), Lemma \ref{tupleinduction}, and the induction hypothesis, if $\beta$ is sufficiently large in terms of $u$, $\bar{c}$, $\lambda$, and $\varepsilon$, then
    \[|E(P_1, ..., P_k)|\leq \frac{\beta}{2}\left(n_1\cdots n_kn_i^{-\frac{1}{c_i}+\varepsilon}+n_iF^\varepsilon_{\bar{c}_{\neq i}}(\bar{n}_{\neq i})\right)<\beta n_iF^\varepsilon_{\bar{c}_{\neq i}}(\bar{n}_{\neq i}).\]
    Therefore, henceforth we suppose $n_1\cdots n_kn_i^{-1/c_i+\varepsilon}\geq n_iF^\varepsilon_{\bar{c}_{\neq i}}(\bar{n}_{\neq i})$ for all $i\in [k]$, whence by Lemma \ref{dumplemma} there is $\tau=\tau(\bar{c},\varepsilon)$ such that, for all $i\in [k]$,
    \begin{equation}\label{betatoalpha}
        n_i F^\varepsilon_{\bar{c}_{\neq i}}(\bar{n}_{\neq i})\leq \tau E_{\bar{c}}(\bar{n})\prod_{i=1}^k n_i^\varepsilon.
    \end{equation}

    Let $\alpha=\alpha(u,\bar{c},\lambda,\varepsilon)$ and $\beta=\beta(u,\bar{c},\lambda,\varepsilon)$ be sufficiently large, to be chosen later. We will show by induction on $n_1+\cdots+n_k$ that (\ref{mainboundalphabeta}) holds.

    Let $m_0\in\N$ such that $m_0>\lambda\delta^{-c_i}(u+1)$ for all $i\in [k]$. If $i\in [k]$ is such that $n_i<m_0$, then for all $i\neq j\in [k]$,
    \[|E(P_1, ..., P_k)|<m_0n_1\cdots n_kn_i^{-1}\leq \beta n_1\cdots n_kn_i^{-1}\leq \beta n_jF^\varepsilon_{\bar{c}_{\neq j}}(n_{\neq j}),\]
    assuming $\beta\geq m_0$. Thus, (\ref{mainboundalphabeta}) holds when $n_1+\cdots+n_k<km_0$, and we henceforth suppose $n_i\geq m_0$ for all $i\in [k]$.

    For $\delta=\delta(\bar{c}, \lambda,\varepsilon)<\frac{1}{4}$ to be chosen later, obtain equipartitions $P_i=A^i_1\sqcup \cdots \sqcup A^i_{K_i}$ as in the definition of $\bar{c}$ as a strong distal regularity tuple, with $\Sigma\subseteq [K_1]\times \cdots\times [K_k]$ the index set of bad cells. By refining the partitions if necessary, we may assume that $|A^i_{j}|\leq 4\delta^{c_i}n_i$ for all $i\in [k]$ and $j\in [K_i]$.
    
    Henceforth, for readability, a tuple $(j_i, ..., j_k)$ is understood to be taken from $[K_i]\times \cdots\times [K_k]$. Let $I_1:=\sum_{(j_1, ..., j_k)\not\in\Sigma}|E(A^1_{j_1}, ..., A^k_{j_k})|$ and $I_2:=\sum_{(j_1, ..., j_k)\in\Sigma}|E(A^1_{j_1}, ..., A^k_{j_k})|$, so that $\abs{E(P_1, ..., P_k)}=I_1+I_2$. We bound $I_1$ and $I_2$.
    
    First, consider $I_1$. For $j_1\in [K_1]$, let $\Sigma_{j_1}:=\{(j_2, ..., j_k): (j_1, ..., j_k)\in\Sigma\}$, so we have
    \[I_1=\sum_{j_1=1}^{K_1} \sum_{(j_2, ..., j_k)\not\in \Sigma_{j_1}}|E(A^1_{j_1}, ..., A^k_{j_k})|.\]
    
    Fix $j_1\in [K_1]$. We have that $|A^1_{j_1}|\geq n_1/K_1-1\geq m_0/(\lambda\delta^{-c_1})-1> u$, so we can fix distinct $a_1, ..., a_u\in A^1_{j_1}$. For $(j_2, ..., j_k)\not\in \Sigma_{j_1}$, $E(A^1_{j_1}, ..., A^k_{j_k})=A^1_{j_1}\times\cdots\times A^k_{j_k}$ or $\emptyset$, and thus
    \[\bigcup_{(j_2, ..., j_k)\not\in \Sigma_{j_1}} E(A^1_{j_1}, ..., A^k_{j_k})\subseteq A^1_{j_1}\times \bigcap_{e=1}^u E(a_e, P_2, ..., P_k).\]
    Now $\bigcap_{e=1}^u E(a_e, P_2, ..., P_k)$ is the induced $(k-1)$-subgraph on $P_2\times \cdots\times P_k$ of the relation $R(x_2, ..., x_k):=\bigwedge_{e=1}^u E(a_e, x_2, ..., x_k)$ on $M$. By Lemma \ref{tupleinduction}, $(c_2, ..., c_k)$ is a strong distal regularity tuple for $R$ with coefficient $u\lambda$. By the induction hypothesis,
    \[\abs{\bigcap_{e=1}^u E(a_e, P_2, ..., P_k)}\ll_{u,\bar{c},\lambda,\varepsilon} F^\varepsilon_{\bar{c}_{\neq 1}}(\bar{n}_{\neq 1}).\]

    Choosing $\beta$ sufficiently large in terms of $u$, $\bar{c}$, $\lambda$, and $\varepsilon$, we can assume that the implied constant is at most $\beta$. Then
    \[I_1\leq \sum_{j_1=1}^{K_1} \beta |A^1_{j_1}|F^\varepsilon_{\bar{c}_{\neq 1}}(\bar{n}_{\neq 1})\leq \beta n_1 F^\varepsilon_{\bar{c}_{\neq 1}}(\bar{n}_{\neq 1}).\]

    Next, consider $I_2$. For each $(j_2, ..., j_k)$, let $B_{j_2, ..., j_k}:=\bigcup_{\substack{1\leq j_1\leq K_1\\(j_2, ..., j_k)\in \Sigma_{j_1}}}A^1_{j_1}$, so we have
    \[I_2=\sum_{(j_2, ..., j_k)}\abs{E(B_{j_2, ..., j_k}, A^2_{j_2}, ..., A^k_{j_k})}.\]
    For each $(j_2, ..., j_k)$, let $s_{j_2, ..., j_k}:=|B_{j_2, ..., j_k}|$. Observe that
    \[\sum_{(j_2, ..., j_k)}s_{j_2, ..., j_k}=\sum_{j_1=1}^{K_1}|A^1_{j_1}| |\Sigma_{j_1}|\leq 4\delta^{c_1}n_1 \sum_{j_1=1}^{K_1}|\Sigma_{j_1}|\leq 4\delta^{c_1}n_1|\Sigma|.\]
    Now observe that
    \[|\Sigma|\leq\frac{\sum_{(j_1, ..., j_k)\in\Sigma}|A^1_{j_1}\timesdots A^k_{j_k}|}{\min_{(j_1, ..., j_k)\in\Sigma}|A^1_{j_1}\timesdots A^k_{j_k}|}\leq \frac{\lambda\delta n_1\cdots n_k}{\prod_{i=1}^k \frac{n_i}{2K_i}}\leq \frac{\lambda\delta n_1\cdots n_k}{\prod_{i=1}^k \frac{1}{2\lambda}\delta^{c_i}n_i}=2^k\lambda^{k+1}\delta^{1-(c_1+\cdots+c_k)},\]
    and so
    \[\sum_{(j_2, ..., j_k)}s_{j_2, ..., j_k}\leq 2^{k+2}\lambda^{k+1}\delta^{1-(c_2+\cdots+c_k)}n_1.\]

    By the induction hypothesis,
    \begin{align*}
    I_2&=\sum_{(j_2, ..., j_k)}\abs{E(B_{j_2, ..., j_k}, A^2_{j_2}, ..., A^k_{j_k})}\\
    &\leq \alpha\sum_{(j_2, ..., j_k)} E_{\bar{c}}(s_{j_2, ..., j_k},4\delta^{c_2}n_2, ..., 4\delta^{c_k}n_k)s_{j_2, ..., j_k}^\varepsilon\prod_{i=2}^k (4\delta^{c_i}n_i)^{\varepsilon}+\beta\sum_{(j_2, ..., j_k)}\sum_{i=1}^k n_iF^\varepsilon_{\bar{c}_{\neq i}}(\bar{n}_{\neq i})\\
    &\leq \alpha\prod_{i=2}^k (4\delta^{c_i}n_i)^{\gamma_i+\varepsilon}\sum_{(j_2, ..., j_k)} s_{j_2, ..., j_k}^{\gamma_1+\varepsilon}+\beta\lambda^{k-1}\delta^{-(c_2+\cdots+c_k)}\sum_{i=1}^k n_iF^\varepsilon_{\bar{c}_{\neq i}}(\bar{n}_{\neq i}).
    \end{align*}
    Recall that $\gamma_1<1$; without loss of generality assume that $\varepsilon<1-\gamma_1$. By H\"older's inequality,
    \begin{align*}
        \sum_{(j_2, ..., j_k)} s_{j_2, ..., j_k}^{\gamma_1+\varepsilon}&\leq \left(\sum_{(j_2, ..., j_k)} s_{j_2, ..., j_k}\right)^{\gamma_1+\varepsilon} (\lambda^{k-1}\delta^{-(c_2+\cdots+c_k)})^{1-\gamma_1-\varepsilon}\\
        &\leq (2^{k+2}\lambda^{k+1}\delta^{1-(c_2+\cdots+c_k)}n_1)^{\gamma_1+\varepsilon} (\lambda^{k-1}\delta^{-(c_2+\cdots+c_k)})^{1-\gamma_1-\varepsilon}\\
        &\leq 2^{2k}\lambda^{2k}\delta^{\gamma_1+\varepsilon-(c_2+\cdots+c_k)}n_1^{\gamma_1+\varepsilon}.
    \end{align*}
    Therefore,
    \begin{align*}
        \prod_{i=2}^k (4\delta^{c_i}n_i)^{\gamma_i+\varepsilon}\sum_{(j_2, ..., j_k)} s_{j_2, ..., j_k}^{\gamma_1+\varepsilon}&\leq 8^{2k}\lambda^{2k}\delta^{\gamma_1+\varepsilon-\sum_{i=2}^k c_i(1-\gamma_i-\varepsilon)}\prod_{i=1}^k n_i^{\gamma_i+\varepsilon}\\
        &=8^{2k}\lambda^{2k}\delta^{(1+c_2+\cdots+c_k)\varepsilon}E_{\bar{c}}(\bar{n})\prod_{i=1}^k n_i^\varepsilon\\
        &\leq \frac{1}{2}E_{\bar{c}}(\bar{n})\prod_{i=1}^k n_i^\varepsilon,
    \end{align*}
    for $\delta$ sufficiently small in terms of $\bar{c}$, $\lambda$, and $\varepsilon$. Thus,
    \[I_2\leq \frac{\alpha}{2}E_{\bar{c}}(\bar{n})\prod_{i=1}^k n_i^\varepsilon+\beta\lambda^{k-1}\delta^{-(c_2+\cdots+c_k)}\sum_{i=1}^k n_iF^\varepsilon_{\bar{c}_{\neq i}}(\bar{n}_{\neq i}).\]
    
    Putting all this together,  
    \begin{alignat*}{2}
        |E(P_1, ..., P_k)|&= I_1+I_2\\
        &\leq \frac{\alpha}{2}E_{\bar{c}}(\bar{n})\prod_{i=1}^k n_i^\varepsilon+\beta(1+\lambda^{k-1}\delta^{-(c_2+\cdots+c_k)})\sum_{i=1}^k n_iF^\varepsilon_{\bar{c}_{\neq i}}(\bar{n}_{\neq i})\\
        &\leq \left(\frac{\alpha}{2}+\tau k\beta(1+\lambda^{k-1}\delta^{-(c_2+\cdots+c_k)})\right) E_{\bar{c}}(\bar{n})\prod_{i=1}^k n_i^\varepsilon&&\;\;\;\;\text{by (\ref{betatoalpha})}\\
        &\leq \alpha E_{\bar{c}}(\bar{n})\prod_{i=1}^k n_i^\varepsilon,
    \end{alignat*}
    for $\alpha$ sufficiently large in terms of $\beta$, $\bar{c}$, $\lambda$, $\delta$, and $\tau$. Thus, (\ref{mainboundalphabeta}) holds as claimed.
\end{proof}
\begin{remark}
    Similarly to Remark \ref{binaryepsilonremark}, it is not hard to see that Theorem \ref{mainabstractbound} remains true if we remove all but one of the occurrences of $\varepsilon$ in each summand of $F^\varepsilon_{\bar{c}}(\bar{n})$, but we will not demonstrate this in detail. Note that Do makes a similar remark \cite[Remark 1.9(ii)]{do}.
\end{remark}
\section{Distality}\label{sectiondistality}
We conclude by describing a broad context in which the distal regularity lemma holds, and hence to which our main theorem can be applied --- distal structures. (However, distal structures are not the only source of relations satisfying the distal regularity lemma --- see Theorem \ref{bays}). This exposition only assumes a basic acquaintance with first-order logic.

Throughout this section, fix a set $M$. For a relation $\phi(x;y)$ on $M$ and $b\in M^y$, write $\phi(M^x; b):=\{a\in M^x: \phi(a,b)\}$. Now, fix a relation $\phi(x;y)$ on $M$, and denote by $\neg\phi(x;y)$ the relation on $M$ that is the complement of $\phi(x;y)$.
\begin{defn}
    For $a\in M^x$ and $B\subseteq M^y$, the \textit{$\phi$-type of $a$ over $B$ in $M$} is
    \[\text{tp}^M_\phi(a/B):=\bigcap_{\substack{b\in B\\ \phi(a;b)}}\phi(M^x; b)\cap \bigcap_{\substack{b\in B\\ \neg\phi(a;b)}}\neg\phi(M^x; b)\subseteq M^x,\]
    and we let $S^M_\phi(B):=\{\text{tp}^M_\phi(a/B): a\in M^x\}$ be the set of $\phi$-types over $B$ in $M$; equivalently, $S^M_\phi(B)$ is the set of Boolean atoms\footnote[2]{Given a set $X$ and a collection $\mathcal{S}=\{S_1, ..., S_n\}$ of subsets of $X$, a \textit{Boolean atom} of $\mathcal{S}$ is a non-empty set of the form $\bigcap_{i\in [n]}S_i^{\varepsilon(i)}$ for some $\varepsilon: [n]\to \{0,1\}$, where $S^1:=S$ and $S^0:=X\setminus S$ for $S\subseteq X$.}  of $\{\phi(M^x; b): b\in M^y\}$. We omit all occurrences of $M$ when it is obvious.
\end{defn}
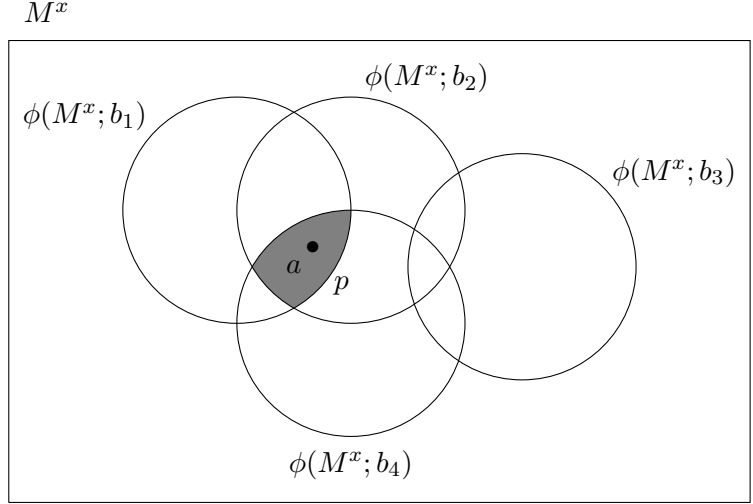
\begin{figure}
\centering
\begin{tikzpicture}[fill=gray]
\scope
\clip (3,4.5) circle (1.5);
\clip (4.5,3) circle (1.5);
\fill (4.5,4.5) circle (1.5);
\endscope
\draw (0,0.625) -- (9.75,0.625) -- (9.75,6.75) -- (0,6.75) -- (0,0.625);
\draw (3,4.5) circle (1.5);
\draw (4.5,4.5) circle (1.5);
\draw (4.5,3) circle (1.5);
\draw (6.75,3.75) circle (1.5);
\node at (1,5.75) {$\phi(M^x; b_1)$};
\node at (5.5,6.25) {$\phi(M^x; b_2)$};
\node at (8.75,5) {$\phi(M^x; b_3)$};
\node at (4.5,1.125) {$\phi(M^x; b_4)$};
\node at (4,4) {\textbullet};
\node at (3.75,3.75) {$a$};
\node at (4.375,3.5) {$p$};
\node at (0.5,7.125) {$M^x$};
\end{tikzpicture}
\caption{The $\phi$-types over $B$}
\label{phitypesfigure}
\end{figure}
We illustrate this definition with Figure \ref{phitypesfigure}, where $B=\{b_1, b_2, b_3, b_4\}$. For $b\in B$, $\phi(M^x; b)$ cuts out a subset of $M^x$, namely, the elements $a\in M^x$ such that $\phi(a;b)$. These form a Venn diagram in the universe $M^x$, whose non-empty regions are precisely the $\phi$-types over $B$, and the set of these is denoted by $S_\phi(B)$, forming a partition of $M^x$. Taking the shaded $\phi$-type $p$ as an example, for any $a\in p$, we have that $p=\text{tp}_\phi(a/B)$ is the $\phi$-type of $a$ over $B$, and every $a'\in p$ has the same $\phi$-type over $B$: for all $b\in B$, $\phi(a;b)$ if and only if $\phi(a';b)$.
\begin{defn}
    The \textit{dual shatter function} of $\phi$, $\pi^*_\phi: \N\to \N$, is such that for all $n\in \N$,
    \[\pi^*_\phi(n)=\max\{|S_\phi(B)|: B\subseteq M^y, |B|=n\}.\]
\end{defn}

It follows by straightforward counting that $\pi^*_\phi(n)\leq 2^n$ for all $n\in\N$. The following theorem, often known as the Sauer--Shelah Lemma, presents a striking dichotomy.
\begin{theorem}
    Either $\pi^*_\phi(n)=2^n$ for all $n\in\N$, or $\pi^*_\phi$ is bounded by a polynomial, that is, there is $d\in\R$ such that $\pi^*_\phi(n)=O(n^d)$ for all $n\in\N$.
\end{theorem}
\begin{proof}
    See, for example, \cite{sauer}.
\end{proof}
That $\pi^*_\phi$ is bounded by a polynomial is one of many equivalent characterisations for $\phi$ to be \textit{NIP (not the independence property)}, a notion that has had innumerable applications in fields such as model theory, combinatorics, and machine learning. Distality refines this notion by requiring the $\phi$-types to be covered `uniformly definably'.
\begin{defn}
    Let $\psi(x; y_1, ..., y_k)$ be a relation on $M$, where $|y_1|=\cdots=|y_k|=|y|$. Say that $\psi$ is a \textit{strong honest definition} for $\phi$ if, for all finite $B\subseteq M^y$ with $|B|\geq 2$\footnote[3]{We require $|B|\geq 2$ so that, roughly speaking, $B$ contains enough elements to `code' information. For instance, it is often convenient to have multiple formulas together performing the role of $\psi$, but passing between one formula and multiple formulas requires $|B|\geq 2$ --- see, for example, \cite[Proposition 2.3]{mypaper}.} and $a\in M^x$, there are $b_1, ..., b_k\in B$ such that $\psi(a; b_1, ..., b_k)$ and $\psi(M^x; b_1, ..., b_k)\subseteq \text{tp}_\phi(a/B)$.
\end{defn}

Note then that the relation $\psi$ induces a \textit{distal (cell) decomposition} for $\phi$: a map $\mathcal{T}$ that assigns, to each finite $B\subseteq M^y$ with $|B|\geq 2$, a cover $\mathcal{T}(B)$ of $M^x$ that refines the partition $S_\phi(B)$, such that $\mathcal{T}(B)\subseteq\{\psi(M^x; b_1, ..., b_k): b_i\in B\}$. The key here is that $\psi$ works for $B\subseteq M^y$ of any finite size at least 2. A $\phi$-type over $B$ can always be defined by a relation of the form
\[\bigwedge_{b\in B'}\phi(x; b)\wedge \bigwedge_{b\in B\setminus B'}\neg\phi(x; b)\]
for some $B'\subseteq B$, but this relation has $|B|$-many parameters. In contrast, the relation $\psi$ has a fixed number $k$ of parameters, and instances of $\psi$ can be used to refine $S_\phi(B)$ for any finite $B\subseteq M^y$ with $|B|\geq 2$.

Note then that $\phi$ is indeed NIP: $S_\phi(B)$ can be refined by a subset of $\{\psi(M^x; b_1, ..., b_k): b_i\in B\}$, so $|S_\phi(B)|\leq |B|^k$.
\begin{defn}
    Let $\mathcal{M}$ be a structure. Say that $\mathcal{M}$ is \textit{distal} if every definable relation $\phi(x;y)$ has a definable (in $\mathcal{M}$) strong honest definition $\psi(x;y_1, ..., y_k)$.
\end{defn}
Note that this is not the original definition of distality as introduced by Simon \cite{simondistal}, but was proven to be equivalent to it by Chernikov and Simon \cite{distaldefn}. Observe that if a structure $\mathcal{M}$ is distal, then it is NIP (that is, every formula in $\mathcal{M}$ is NIP).

The class of distal structures includes the class of \textit{o-minimal} structures, such as the real ordered field $(\R,+,\times,<)$ and its expansion $(\R,+,\times,<,\text{exp})$, where $\text{exp}: \R\to\R$ is the exponential function. That o-minimal structures are distal was shown by Simon \cite{simondistal}. As mentioned in Section \ref{sectionintro}, there are also many distal structures that are not o-minimal. One example is the p-adics $(\mathbb{Q}_p, +, \cdot, v(x)\geq v(y))$, shown by Simon \cite{simondistal} to be distal. Another example is Presburger arithmetic $(\Z,+,<)$, the distality of which follows quickly from standard quantifier elimination results (or see \cite{regularitylemma} for an alternative proof). In \cite{mypaper}, we proved that distality is preserved when $(\Z,+,<)$ is expanded by a `congruence-periodic' and `sparse' predicate $R\subseteq \N$, such as $\{2^n: n\in \N\}$, $\{n!: n\in \N\}$, and the set of Fibonacci numbers.

We had previously wondered if distal structures were the only source of relations satisfying the distal regularity lemma. That is, if $\phi(x_1, ..., x_k)$ is a relation on a set $M$ satisfying the distal regularity lemma, must the structure $(M,\phi)$ admit a distal expansion? The answer is no: we are grateful to Martin Bays for suggesting the following counterexample to us in personal communication.
\begin{theorem}\label{bays}
    Let $K$ be a finitely generated extension of $\mathbb{F}_p$, such as $K=\mathbb{F}_p(t)$, and let $\phi(x,y;m,c):=(y=mx+c)$ be the point-line incidence relation. Then $\phi$ satisfies the distal regularity lemma as a relation on $K$, but the structure $(K,\phi)$ does not admit a distal expansion.
\end{theorem}
\begin{proof}
    We first argue that $\phi$ satisfies the distal regularity lemma as a relation on $K$. By \cite[Lemma 4.1]{bays}, $K$ admits a valuation $v$ with finite residue field, so we may view $K$ as a structure $\mathcal{K}$ over the language $\{+,\times,\leq\}$ of valued fields, where $x\leq  y:\Leftrightarrow v(x)\leq v(y)$ in $\mathcal{K}$. Let $\mathcal{K}^*$ be an algebraically closed valued field such that $\mathcal{K}\subseteq \mathcal{K}^*$; it is folklore that $\mathcal{K}^*$ is NIP.
    
    By \cite[Theorem 5.6]{bays}, as a relation on $K$, $\phi$ has a strong honest definition $\psi(x,y;m_1,c_1, ..., m_k,c_k)$ definable in $\mathcal{K}^*$, in the following sense: for all finite $B\subseteq K^2$ with $|B|\geq 2$ and $a\in K^2$, there are $b_1, ..., b_k\in B$ such that $\mathcal{K}^*\models \psi(a; b_1, ..., b_k)$ and, for all $a'\in (\mathcal{K}^*)^2$, if $\mathcal{K}^*\models \psi(a';b_1, ..., b_k)$ then $\mathcal{K}^*\models \phi(a;b')\leftrightarrow \phi(a';b')$ for all $b'\in B$. The proof of \cite[Lemma 3.6]{cuttinglemma} now gives a `cutting lemma' for $\phi$ as follows. For all finite $B\subseteq K^2$ with $|B|\geq 2$ and $\delta\in (0,1)$, there is a cover $\mathcal{F}\subseteq \{\psi((\mathcal{K}^*)^2; b_1, ..., b_k): b_i\in B\}$ of $(\mathcal{K}^*)^2$, such that $|\mathcal{F}|\leq \poly_{\phi,\psi}(\delta^{-1})$ and for all $F\in\mathcal{F}$,
    \[\#\{b\in B: F\subseteq \phi((\mathcal{K}^*)^2;b)\text{ or }F\subseteq \neg\phi((\mathcal{K}^*)^2;b)\}\geq (1-\delta)|B|.\]

    Let $P,Q\subseteq K^2$ be finite with $|Q|\geq 2$ and $\delta\in (0,1)$; we give appropriate partitions of $P$ and $Q$ to show that $\phi$ satisfies the distal regularity lemma. Applying the cutting lemma above with $B=Q$, we obtain a cover $\mathcal{F}\subseteq \{\psi((\mathcal{K}^*)^2; b_1, ..., b_k): b_i\in Q\}$ of $(\mathcal{K}^*)^2$. For all $F\in\mathcal{F}$ and $\sigma\in\{0,1\}$, let $D_F^\sigma:=\{d\in (\mathcal{K}^*)^2: F\subseteq \phi^\sigma((\mathcal{K}^*)^2;d)\}$. Let $\mathcal{G}$ be the set of Boolean atoms of $\{D_F^\sigma: F\in\mathcal{F}, \sigma\in\{0,1\}\}$, so $\mathcal{G}$ is a partition of $(\mathcal{K}^*)^2$. Since $\mathcal{K}^*$ is NIP, $|\mathcal{G}|\leq\poly_{\phi,\psi,\mathcal{K}^*}(|\mathcal{F}|)$, and so $|\mathcal{G}|\leq \poly_{\phi,\psi,\mathcal{K}^*}(\delta^{-1})$. Let $\mathcal{F}_0$ be any partition of $(\mathcal{K}^*)^2$ refining $\mathcal{F}$ such that $|\mathcal{F}_0|\leq |\mathcal{F}|\leq \poly_{\phi,\psi}(\delta^{-1})$. The reader is invited to check that the partitions $\mathcal{F}_0\cap P$ and $\mathcal{G}\cap Q$ are such that $\sum |F\times G|\leq \delta |P||Q|$, where the sum ranges over all $(F,G)\in (\mathcal{F}_0\cap P)\times (\mathcal{G}\cap Q)$ such that $F\times G$ is not $\phi$-homogeneous.

    It remains to argue that $(K,\phi)$ does not admit a distal expansion. Let $\overline{\mathcal{K}}=(K,+,\times)$ be the field structure on $K$. Now $K\not\supseteq\mathbb{F}_p^{\text{alg}}$ since every sub-extension of a finitely generated field extension is finitely generated (see, for example, \cite[Theorem 24.9]{algebratextbook}), so $\overline{\mathcal{K}}$ is not NIP by \cite[Corollary 4.5]{kaplanscanlonwagner}. Thus, there is a formula $\psi$ in $\overline{\mathcal{K}}$ that is not NIP. Now, the field operations $+$ and $\times$ are definable in $(K,\phi)$: indeed, $0$ and $1$ are $\emptyset$-definable in $(K,\phi)$, and for all $p,q,r\in K$, $p+q=r\Leftrightarrow \phi(p,r;1,q)$ and $p\times q=r\Leftrightarrow \phi(p,r;q,0)$. Thus, $\psi$ is definable in $(K,\phi)$ and all of its expansions. We conclude that every expansion of $(K,\phi)$ is not NIP and hence not distal.
\end{proof}
\vspace{-7pt}
\bibliographystyle{plainurl}
\bibliography{bib}

@article{cuttinglemma,
    author = {Artem Chernikov and David Galvin and Sergei Starchenko},
    title = "{Cutting lemma and Zarankiewicz’s problem in distal structures}",
    journal = {Selecta Mathematica},
    volume = {26},
    number = {25},
    pages = {471-508},
    year = {2020},
    doi = {10.1007/s00029-020-0551-2},
    url = {https://doi.org/10.1007/s00029-020-0551-2},
}

@article{fox,
    author    = "Jacob Fox and J\'{a}nos Pach and Adam Sheffer and Andrew Suk and Joshua Zahl",
    title     = "A semi-algebraic version of {Zarankiewicz's} problem",
    journal   = "Journal of the European Mathematical Society",
    volume   = "19",
    number   = "6",
    pages    = "1785--1810",
    year      = "2017",
doi = "10.4171/JEMS/705"
}

@inproceedings{szemerediregularitylemma,
    author = {E. Szemer\'edi},
    title = {Regular partitions of graphs},
    booktitle = {Probl\`emes combinatoires et
th\'eorie des graphes (Univ. Orsay, Orsay, 1976)},
series = {Colloq. Internat. CNRS},
volume = {260},
    year = {1978},
    pages = {399--401}
}

@article{mypaper,
title={Distal expansions of {Presburger} arithmetic by a sparse predicate},
DOI={10.1017/jsl.2025.10125},
journal={The Journal of Symbolic Logic},
author={Mervyn Tong},
year={2025},
pages={1--26}}

@article{regularitylemma,
      title={Regularity lemma for distal structures}, 
      author={Artem Chernikov and Sergei Starchenko},
      year={2018},
    journal = {Journal of the European Mathematical Society},
    volume = {20},
    number = {10},
    pages = {2437--2466},
doi = {10.4171/JEMS/816}
}

@article{do,
title = {Zarankiewicz's problem for semi-algebraic hypergraphs},
journal = {Journal of Combinatorial Theory, Series A},
volume = {158},
pages = {621--642},
year = {2018},
issn = {0097-3165},
doi = {10.1016/j.jcta.2018.04.007},
author = {Do, Thao T.},
}

@article{kovarisosturan,
author = {K\H{o}v\'ari, T. and S\'os, V. T. and Tur\'an, P.},
journal = {Colloquium Mathematicae},
language = {eng},
number = {1},
pages = {50-57},
title = {{On a problem of K. Zarankiewicz}},
volume = {3},
year = {1954},
doi = {10.4064/cm-3-1-50-57}
}

@article{erdoshypergraphs,
title = {On extremal problems of graphs and generalized graphs},
journal = {Israel Journal of Mathematics},
volume = {2},
number = {3},
pages = {183--190},
year = {1964},
author = {Erd\H{o}s, Paul},
doi = {10.1007/BF02759942}
}

@article{distaldefn,
 ISSN = {00029947},
 author = {Artem Chernikov and Pierre Simon},
 journal = {Transactions of the American Mathematical Society},
 number = {7},
 pages = {5217--5235},
 publisher = {American Mathematical Society},
 title = {Externally definable sets and dependent pairs {II}},
 urldate = {2023-10-05},
 volume = {367},
 year = {2015},
doi = {10.1090/S0002-9947-2015-06210-2}
}

@article{simondistal,
    title = {Distal and non-distal {NIP} theories},
    journal = {Annals of Pure and Applied Logic},
    volume = {164},
    number = {3},
    pages = {294--318},
    year = {2013},
    issn = {0168-0072},
    doi = {https://doi.org/10.1016/j.apal.2012.10.015},
    author = {Pierre Simon}}

@article{semialgebraicregularity,
author = {Fox, Jacob and Pach, J\'{a}nos and Suk, Andrew},
title = {A Polynomial Regularity Lemma for Semialgebraic Hypergraphs and Its Applications in Geometry and Property Testing},
journal = {SIAM Journal on Computing},
volume = {45},
number = {6},
pages = {2199-2223},
year = {2016},
doi = {10.1137/15M1007355},
}

@article{gowers1997,
  title={{Lower bounds of tower type for Szemer{\'e}di's uniformity lemma}},
  author={Gowers, William T.},
  journal={Geometric \& Functional Analysis GAFA},
  year={1997},
  volume={7},
  pages={322-337},
doi = {10.1007/PL00001621}
}

@article{sauer,
title = {On the density of families of sets},
journal = {Journal of Combinatorial Theory, Series A},
volume = {13},
number = {1},
pages = {145-147},
year = {1972},
issn = {0097-3165},
doi = {10.1016/0097-3165(72)90019-2},
author = {N. Sauer},
}

@article{malliarisshelah,
 ISSN = {00029947},
 author = {M. Malliaris and S. Shelah},
 journal = {Transactions of the American Mathematical Society},
 number = {3},
 pages = {1551--1585},
 publisher = {American Mathematical Society},
 title = {Regularity lemmas for stable graphs},
 urldate = {2024-10-17},
 volume = {366},
 year = {2014},
doi = {10.1090/S0002-9947-2013-05820-5}
}

@article{nipregularitylemma,
 author = {Jacob Fox and J\'{a}nos Pach and Andrew Suk},
 journal = {Discrete \& Computational Geometry},
 pages = {809--829},
 title = {{Erd\H{o}s--Hajnal Conjecture for Graphs with Bounded VC-Dimension}},
 volume = {61},
 year = {2019},
doi = {10.4230/LIPIcs.SoCG.2017.43}
}

@misc{hans,
      title={{Multilevel polynomial partitioning and semialgebraic hypergraphs: regularity, Tur\'an, and Zarankiewicz results}}, 
      author={Jonathan Tidor and Hung-Hsun Hans Yu},
      year={2024},
      archivePrefix={arXiv},
      primaryClass={math.CO},
      eprint={2407.20221}, 
}

@article{bays,
     author = {Bays, Martin and Martin, Jean-Fran\c{c}ois},
     title = {Incidence bounds in positive characteristic via valuations and distality},
     journal = {Annales Henri Lebesgue},
     pages = {627--641},
     publisher = {\'ENS Rennes},
     volume = {6},
     year = {2023},
     doi = {10.5802/ahl.174},
     language = {en},
}

@article{kaplanscanlonwagner,
     author = {Itay Kaplan and Thomas Scanlon and Frank O. Wagner},
     title = {{Artin-Schreier extensions in NIP and simple fields}},
     journal = {Israel Journal of Mathematics},
     pages = {141--153},
     volume = {185},
     year = {2011},
     doi = {10.1007/s11856-011-0104-7},
     url = {https://doi.org/10.1007/s11856-011-0104-7}
}

@book{algebratextbook,
author = {Isaacs, I. Martin},
isbn = {9780821847992},
publisher = {American Mathematical Society},
series = {Graduate Studies in Mathematics},
volume = {100},
title = {Algebra: A Graduate Course},
year = {1994},
}

@article{janzerpohoata,
     author = {Oliver Janzer and Cosmin Pohoata},
     title = {{On the Zarankiewicz Problem for Graphs with Bounded VC-Dimension}},
     journal = {Combinatorica},
     pages = {839--848},
     volume = {44},
     year = {2024},
     doi = {10.1007/s00493-024-00095-2},
     url = {https://doi.org/10.1007/s00493-024-00095-2}
}
\end{document}